\documentclass[11pt]{amsart}
\usepackage{amssymb, latexsym}
\theoremstyle{plain}
\newtheorem{theorem}{Theorem}

\newtheorem{proposition}{Proposition}

\theoremstyle{remark}

\newtheorem*{Remark 1}{Remark 1}
\newtheorem*{Remark 2}{Remark 2}
\newtheorem*{Remark 3}{Remark 3}
\newtheorem*{Remark 4}{Remark 4}

\numberwithin{equation}{section}

\begin{document}

\title[]
 {Transience/Recurrence and Growth Rates for Diffusion Processes in Time-Dependent Domains}

\author{Ross G. Pinsky}
\address{Department of Mathematics\\
Technion---Israel Institute of Technology\\
Haifa, 32000\\ Israel}
\email{ pinsky@math.technion.ac.il}
\urladdr{http://www.math.technion.ac.il/~pinsky/}

\date{}

\begin{abstract}
Let $\mathcal{K}\subset R^d$, $d\ge2$,
be a smooth, bounded domain  satisfying $0\in\mathcal{K}$, and let $f(t),\ t\ge0$, be a  smooth, continuous, nondecreasing
function satisfying $f(0)>1$. Define $D_t=f(t)\mathcal{K}\subset R^d$. Consider a diffusion process
corresponding to the generator $\frac12\Delta+b(x)\nabla$ in the time-dependent domain $D_t$   with normal reflection
at the time-dependent boundary.
Consider also the one-dimensional diffusion process corresponding to
the generator $\frac12\frac{d^2}{dx^2}+B(x)\frac d{dx}$ on the time-dependent domain $(1,f(t))$ with reflection
 at the boundary.
We give precise conditions for transience/recurrence
of the one-dimensional process in terms of the growth rates of $B(x)$ and $f(t)$.
In the recurrent case, we also investigate positive recurrence, and in the transient case,
we also consider the asymptotic growth rate  of the process.
Using the one-dimensional results, we
give conditions for transience/recurrence
of the multi-dimensional process in terms of the growth rates of $B^+(r)$, $B^-(r)$ and $f(t)$,
where
$B^+(r)=\max_{|x|=r}b(x)\cdot\frac x{|x|}$ and
$B^-(r)=\min_{|x|=r}b(x)\cdot\frac x{|x|}$.

\end{abstract}

\maketitle

\section{Introduction and Statement of  Results}
Let $\mathcal{K}\subset R^d$, $d\ge2$,
be a bounded domain with $C^3$-boundary satisfying $0\in\mathcal{K}$, and let $f(t),\ t\ge0$, be a  continuous, nondecreasing
$C^3$-function satisfying $f(0)>1$. Define $D_t=f(t)\mathcal{K}\subset R^d$.
It is known that one can define a Brownian motion $X(t)$ with normal reflection at the boundary in the time-dependent
domain $\{(x,t): x\in D_t, t\ge0\}$. More precisely, one has for $0\le s<t$,
$$
\begin{aligned}
&X(t)=x+W(t)-W(s)+\int_s^t 1_{\partial D_u}(X(u))n(u,X(u))d\mathcal{L}_u,\\
&\mathcal{L}_t=\int_s^t1_{\partial D_u}(X(u))d\mathcal{L}_u,
\end{aligned}
$$
where $W(\cdot)$ is a Brownian motion, $n(u,x)$ is the unit inward normal to $D_u$ at $x\in\partial D_u$ and
 $\mathcal{L}_u$ is the local time up to time $u$ of $X(\cdot)$ at the time-dependent boundary.
See \cite{BCS}.

The  process $X(t)$ is  \it recurrent\rm\ if,  with probability one,
$X(t)\in \mathcal{K}$ at arbitrarily large times $t$, and is
\it transient\rm\ if, with probability zero,
$X(t)\in \mathcal{K}$ at arbitrarily large times $t$.
 As with non-degenerate diffusion processes in unrestricted space, transience is equivalent to
$\lim_{t\to\infty}|X(t)|=\infty$ with probability one.
It is simple to see that the definitions are independent of the starting point and the starting time of the process.
In a recent paper \cite{DHS}, it was  shown
that for $d\ge3$, if $\int^\infty \frac1{f^d(t)}dt<\infty$, then  the process is transient, while if
$\int^\infty \frac1{f^d(t)}dt=\infty$, and an additional technical condition is fulfilled, then the process is  recurrent. The additional technical condition is that either $\mathcal{K}$ is a ball, or that
$\int_0^\infty (f')^2(t)dt<\infty$.
In particular, this result indicates that if for sufficiently large $t$, $f(t)=ct^a$, for some $c>0$, then the
 process is transient if $a>\frac 1d$ and recurrent if $a\le\frac1d$.
 The paper \cite{DHS} also
studies the analogous problem for simple, symmetric random walk in growing domains.

In this paper we study the transience/recurrence dichotomy in the case that the Brownian motion is replaced
by a diffusion process; namely,  Brownian motion with a locally bounded drift $b(x)$. That is, the  generator of the process
when it is away from the boundary is $\frac12\Delta+b(x)\nabla$ instead of $\frac12\Delta$.
Using the Cameron-Martin-Girsanov change-of-measure formula, or alternatively in the case of a Lipschitz drift, by a direct
construction as in \cite{BCS}, one can show that the diffusion process in the time-dependent domain
can be defined.
We will show how the strength of the radial component,  $b(x)\cdot \frac{x}{|x|}$, of the drift,
and the growth rate of the domain--via $f(t)$--affect the transience/recurrence dichotomy.

In fact, we will prove a transience/recurrence dichotomy for  a one-dimensional process.
Our result for the multi-dimensional case will follow readily from the one-dimensional result along
with results in \cite{DHS}.
 Let $f(t)$ be as in the first paragraph.
Consider  the diffusion process corresponding to the generator
 $\frac12\frac{d^2}{dx^2}+B(x)\frac d{dx}$, where $B$ is locally bounded, in the time-dependent domain $(1,f(t))$ with reflection at the
 endpoint $x=1$ (for all times) and
at the endpoint $f(t)$ at time $t$.
If $B(x)=\frac kx$, the process is a Bessel process. When this process is considered on the space $(1,\infty)$ with
reflection at 1, it is recurrent for $k\le \frac12$ and transient for $k>\frac12$. In particular, it is the radial part
of a $d$-dimensional Brownian motion when $k=\frac{d-1}2$. The result of \cite{DHS} noted above can presumably be slightly modified
to show that for $k>\frac12$, the process on the time dependent domain $(1,f(t))$ with reflection at the endpoints
is transient or recurrent according to whether
$\int^\infty \frac1{f^{2k+1}(t)}dt<\infty$ or $\int^\infty \frac1{f^{2k+1}(t)}dt=\infty$.
In this paper we considers drifts that are on a larger order than $\frac1x$.
We will prove the following theorem concerning transience/recurrence.
\begin{theorem}\label{1}
Consider the diffusion process corresponding to the generator
  $\frac12\frac{d^2}{dx^2}+B(x)\frac d{dx}$ in the time-dependent domain $(1,f(t))$, with reflection at
both the fixed endpoint and the time-dependent one.
Let $\gamma>-1$ and $b,c>0$.

\noindent i. Assume that
$$
\begin{aligned}
&B(x)\le bx^\gamma,\ \text{ for sufficiently large}\ x,\\
& f(t)\le c(\log t)^\frac1{1+\gamma},\
\text{for sufficiently large}\ t.
\end{aligned}
$$
 If
$$
\frac{2bc^{1+\gamma}}{1+\gamma}<1,\ \ \ \ \
\text{or}\ \ \
\frac{2bc^{1+\gamma}}{1+\gamma}=1 \ \text{and}\ \gamma\ge-\frac12,
$$
then the process is recurrent.

\noindent ii. Assume that
$$
\begin{aligned}
&B(x)\ge bx^\gamma,\ \text{ for sufficiently large}\ x,\\
& f(t)\ge c(\log t)^\frac1{1+\gamma},\
\text{for sufficiently large}\ t.
\end{aligned}
$$
 If
$$
\frac{2bc^{1+\gamma}}{1+\gamma}> 1,
$$
then the process is transient.

\end{theorem}
\noindent\bf Remark.\rm\ We expect that the process is also recurrent in part (i)
if $\frac{2bc^{1+\gamma}}{1+\gamma}=1$ and $\gamma\in(-1,-\frac12)$.
\medskip

Using Theorem \ref{1}, we will  prove the following result for the multi-dimensional process.

\begin{theorem}\label{2}
Consider the diffusion process corresponding to the generator
 $\frac12\Delta+b(x)\nabla$ in the time-dependent domain $D(t)=f(t)\mathcal{K}$,
where $\mathcal{K}$ and $f$ are as in the first paragraph.
Let
$$
B^+(r)=\max_{|x|=r}b(x)\cdot\frac x{|x|},\  \ \
B^-(r)=\min_{|x|=r}b(x)\cdot\frac x{|x|},
$$
and let
$$
\text{rad}\thinspace^+(\mathcal{K})=\max(|x|:x\in\partial \mathcal{K}),\ \ \
\text{rad}\thinspace^-(\mathcal{K})=\min(|x|:x\in\partial \mathcal{K}).
$$
Let $\gamma>-1$ and $b,c>0$.

\noindent i. Assume that
\begin{equation}\label{recurcondition}
\begin{aligned}
&B^+(r)\le b r^\gamma,\ \text{ for sufficiently large} \ r,\\
&f(t)\le \frac c{\text{rad}\thinspace^+(\mathcal{K})}(\log t)^{\frac1{1+\gamma}},\ \text{ for sufficiently
  large}\  t.
\end{aligned}
\end{equation}
Also assume either that $\mathcal{K}$ is a ball or that
  $\int_0^\infty (f')^2(t)dt<\infty$.
\newline
  If
$$
\frac{2bc^{1+\gamma}}{1+\gamma}<1,\ \ \ \ \
\text{or}\ \ \
\frac{2bc^{1+\gamma}}{1+\gamma}=1, \ d=2\  \text{and}\ \gamma\ge0,
$$
then the process is recurrent.

\noindent ii. Assume that
\begin{equation}\label{transcondition}
\begin{aligned}
& B^-(r)\ge b r^\gamma,\ \text{ for sufficiently large}\ r,\\
& f(t)\ge \frac c{\text{rad}\thinspace^-(\mathcal{K})}(\log t)^{\frac1{1+\gamma}},\ \text{ for sufficiently
  large}\ t.
\end{aligned}
\end{equation}
 If
$$
\frac{2bc^{1+\gamma}}{1+\gamma}>1,
$$
then the process is transient.
\end{theorem}
\noindent \bf Remark 1.\rm\ We expect that the process is  recurrent in part (i)
when $\frac{2bc^{1+\gamma}}{1+\gamma}=1$, for all values of $\gamma>-1$ and $d\ge2$.
\medskip

\noindent \bf Remark 2.\rm\ If $f(t)=C(\log t)^{\frac1{1+\gamma}}$, for all large $t$,
where   $C>0$ and $\gamma>-1$, then the condition
$\int_0^\infty (f')^2(t)dt<\infty$ in part (i) is satisfied.

\medskip

In the recurrent case, it is natural to consider
 \it positive recurrence\rm, which we define as follows:
the one-dimensional process above is  positive recurrent
if starting from $x>1$, the expected value of the first hitting time of 1 is finite, while the multi-dimensional
 process defined above is positive recurrent if starting from a point  $x\not\in\bar{\mathcal{K}}$,
the expected value of the first hitting time of $\bar{\mathcal{K}}$ is finite. It is simple to see that
this definition is independent of the starting point and the starting time of the process.
We have the following theorem regarding positive recurrence of the one-dimensional process.

\begin{theorem}\label{3}
Under the conditions of part (i) of  Theorem \ref{1} or Theorem \ref{2}, the process is positive recurrent if
$$
\frac{2bc^{1+\gamma}}{1+\gamma}<1.
$$
\end{theorem}

\bf\noindent Remark.\rm\ The proof of Theorem \ref{3} relies heavily on the estimates
in the proof of part (i) of Theorem \ref{1}. We suspect that in
the borderline cases, when $\frac{2bc^{1+\gamma}}{1+\gamma}=1$, the process
is never positive recurrent. However, the estimates in the proof of part (ii) of Theorem \ref{1}
don't go quite far enough to prove this.
\medskip

In the transient case, it is natural to consider the asymptotic growth rate of the process.
 It is known that
the process $X(t)$ corresponding to the generator
  $\frac12\frac{d^2}{dx^2}+bx^\gamma\frac d{dx}$ on $[1,\infty)$ with reflection at 1
  grows a.s. on the order
  $t^{\frac1{1-\gamma}}$ if  $\gamma\in(-1,1)$.
  (In fact, the solutions $\hat x(t)$ to the differential equation $x'=bx^\gamma$
satisfy $\lim_{t\to\infty}\frac{\hat x(t)}{t^\frac1{1-\gamma}}=(b(1-\gamma))^\frac1{1-\gamma}$, and it is not hard
to show that $X(t)$ satisfies $\lim_{t\to\infty}\frac{X(t)}{t^{\frac1{1-\gamma}}}=(b(1-\gamma)^\frac1{1-\gamma}$ a.s.)
  The process  grows a.s.
  exponentially if $\gamma=1$, and explodes  a.s. if $\gamma>1$ \cite{P87}.
  From this it is clear that the one-dimensional process $X(t)$
  with $B(x)=bx^\gamma$ on the time-dependent domain $(1,f(t))$ satisfies
  \begin{equation*}
  X(t)=f(t)\ \text{for arbitrarily large}\ t \ \text{a.s.},
  \end{equation*}
and consequently,
  \begin{equation}\label{limsup}
  \limsup_{t\to\infty}\frac{X(t)}{f(t)}=1\ \text{a.s.},
  \end{equation}
   if $f(t)=o(t^{\frac1{1-\gamma}})$
  and $\gamma\in(-1,1)$, if $f(t)$ grows sub-exponentially  and $\gamma=1$,
  and with no restrictions on  $f$ if $\gamma>1$.
The next theorem treats the  behavior of $\liminf_{t\to\infty}\frac{X(t)}{f(t)}$ in what turns out to be the delicate case that
$B(x)=bx^\gamma$ and
$f(t)=c(\log t)^{\frac1{1+\gamma}}$, with $\frac{2bc^{1+\gamma}}{1+\gamma}>1$.
(Recall from Theorem \ref{1} that if $\frac{2bc^{1+\gamma}}{1+\gamma}<1$, then the process is recurrent
and thus $\liminf_{t\to\infty}X(t)=1$.)
We restrict to $\gamma\in(-1,1)$ for technical reasons, but we suspect that the following result also holds
for $\gamma\ge1$.
\begin{theorem}\label{liminf}
Consider the diffusion process corresponding to the generator
  $\frac12\frac{d^2}{dx^2}+B(x)\frac d{dx}$ in the time-dependent domain $(1,f(t))$, with reflection at
both the fixed endpoint and the time-dependent one.
Let $\gamma\in(-1,1)$ and $b,c>0$. Assume that for sufficiently large $x,t$,
$$
\begin{aligned}
&B(x)=bx^\gamma,\\
&f(t)=c(\log t)^{\frac1{1+\gamma}},
\end{aligned}
$$
where
$$
\frac{2bc^{1+\gamma}}{1+\gamma}>1.
$$
Then
$$
\liminf_{t\to\infty}\frac{X(t)}{f(t)}=\big(1-\frac{1+\gamma}{2bc^{1+\gamma}}\big)^\frac1{1+\gamma}\ \text{a.s.}
$$
\end{theorem}
\medskip

 We now consider the asymptotic growth behavior
in the case that  $B(x)=x^\gamma$, $\gamma\in(-1,1)$, and that $f(t)$ is on a larger order than $(\log t)^{\frac1{1+\gamma}}$, but
 on a smaller order than $t^{\frac1{1-\gamma}}$. (Recall from the paragraph preceding Theorem \ref{liminf}
  that this latter order is the order
on which the process would grow if it lived on $[1,\infty)$ rather than on the
time-dependent domain.)  For simplicity we will assume that
$f(t)=(\log t)^l$, with $l>\frac1{1+\gamma}$, or that $f(t)=t^l$, with
$l\in(0,\frac1{1-\gamma})$. (We have dispensed  with the coefficients $b$ and $c$ because here they no longer play
a role at the level of  asymptotic behavior we investigate.)
\begin{theorem}\label{liminfagain}
Consider the diffusion process corresponding to the generator
  $\frac12\frac{d^2}{dx^2}+B(x)\frac d{dx}$ in the time-dependent domain $(1,f(t))$, with reflection at
both the fixed endpoint and the time-dependent one.
Let $\gamma\in(-1,1)$. Assume that
$$
B(x)=x^\gamma.\\
$$

\noindent i.
 Assume  that for $t\ge2$,
$$
f(t)=(\log t)^l, \ \text{with}\ l>\frac1{1+\gamma}.
$$
Then
$$
\lim_{t\to\infty}\frac{X(t)}{f(t)}=1\ \text{a.s.}
$$

\noindent ii. Assume  that
$$
f(t)=t^l, \ \text{with}\ l\in(0,\frac1{1-\gamma}).
$$
Let
$$
q_0=\begin{cases} 0,\ \text{if}\ \gamma\ge0;\\ -l\gamma,\ \text{if}\ \gamma\in(-1,0).\end{cases}
$$
Then
\begin{equation}\label{q0role}
\limsup_{t\to\infty}\frac{f(t)-X(t)}{t^q}=0\ \text{a.s.}\ \text{for }\ q>q_0,
\end{equation}
and
\begin{equation}\label{q0roleagain}
\limsup_{t\to\infty}\frac{f(t)-X(t)}{t^{q_0}}=\infty\ \text{a.s.}, \ \text{when} \ \gamma\in(-1,0].
\end{equation}
In particular  (in light of \eqref{limsup}),
$$
\lim_{t\to\infty}\frac{X(t)}{f(t)}=1\ \text{a.s.}
$$

\end{theorem}

\noindent \bf Remark.\rm\ Note in particular that for $b(x)=x^\gamma$ and $f(t)=t^l$,
if $\gamma\in[0,1)$, then the deviation of $X(t)$  from $f(t)$ as $t\to\infty$ is $o(t^q)$, for any $q> 0$, while if
$\gamma\in(-1,0)$, then this deviation is $o(t^q)$ for $q>-l\gamma$, but not for $q=-l\gamma$.

\medskip

Asymptotic growth behavior in the spirit of Theorems \ref{liminf} and \ref{liminfagain}
for the multi-dimensional case can be gleaned just as Theorem \ref{2} was gleaned from Theorem \ref{1}.

In section 2 we prove several auxiliary
results which will be needed for the proofs. The proofs of Theorem \ref{1}-\ref{liminfagain} are given in sections 3-7 respectively.

Throughout the paper, the following notation will be employed:

\noindent Let $X(t)$ denote a canonical, continuous real-valued path, and let
$T_\alpha=\inf\{t\ge0:X(t)=\alpha\}$.
Let
$$
L_{bx^\gamma}=\frac12\frac{d^2}{dx^2}+bx^\gamma\frac d{dx}.
$$
Let $P_x^{bx^\gamma;\text{Ref}\leftarrow:\beta}$ and
$E_x^{bx^\gamma;\text{Ref}\leftarrow:\beta}$  denote probabilities and expectations for the diffusion process
corresponding to $L_{bx^\gamma}$ on $[1,\beta]$, starting from $x\in[1,\beta]$, with reflection at $\beta$ and stopped at 1,
and let $P_x^{bx^\gamma;\text{Ref}\rightarrow:\alpha}$ and $E_x^{bx^\gamma;\text{Ref}\rightarrow:\alpha}$
 denote probabilities and expectations for the diffusion process
corresponding to $L_{bx^\gamma}$  on $[\alpha,\infty)$, starting from $x\in[\alpha,\infty)$, with reflection at $\alpha$.
We note that this latter diffusion is explosive if $\gamma>1$, but we will only be considering it until
time $T_\beta$ for some $\beta>\alpha$.
We will sometimes work with a constant drift, which we will denote by  $D$ (instead of $bx^\gamma$ with $\gamma=0$),
in which case $D$ will replace $bx^\gamma$ in all of the above notation.

\section{Auxiliary  Results}

In this section we prove four propositions. The first three of them are used explicitly
in the proof of Theorem \ref{1}, and implicitly in many of the other theorems, since
many of the calculations in the proof of Theorem \ref{1} are used in the proofs of the other theorems.
Proposition \ref{4} is  used only for the proof of \eqref{q0roleagain} in Theorem \ref{liminfagain}.

\begin{proposition}\label{1}
For $\alpha\in[1,\beta]$,
\begin{equation}\label{est-mgfalphabeta}
E_x^{bx^\gamma;\text{Ref}\leftarrow:\beta}\exp (\lambda T_\alpha)\le 2,\ \text{for}\ x\in[\alpha, \beta]\
  \text{and}\ \lambda\le \hat\lambda(\alpha,\beta),
\end{equation}
where
\begin{equation}\label{hatlambda}
\hat\lambda(\alpha,\beta)=\exp\Big(-\big(2+2b\max(\alpha^\gamma,\beta^\gamma)\big)(\beta-\alpha)\Big).
\end{equation}
\end{proposition}
\begin{proof}
Of course, it suffices to work with $\lambda\ge0$.
Consider the function
\begin{equation}\label{u}
u(x)=2-\exp(-r(x-\alpha)), \  \alpha\le x\le \beta,
\end{equation}
where $r>0$.
Then
\begin{equation}\label{evestimate}
\exp(r(x-\alpha))(L_{bx^\gamma}+\lambda)u=-\frac12r^2+rbx^\gamma-\lambda +2\lambda \exp(r(x-\alpha)),\ x\in[\alpha,\beta].
\end{equation}
For $\lambda\ge0$,
$$
\begin{aligned}
&\sup_{x\in[\alpha,\beta]}\big(-\frac12r^2+rbx^\gamma-\lambda +2\lambda \exp(r(x-\alpha))\big)\le\\
& -\frac12r^2+rb\max(\alpha^\gamma,\beta^\gamma)-\lambda +2\lambda \exp(r(\beta-\alpha)).
\end{aligned}
$$
Thus, we have
$(L_{bx^\gamma}+\lambda)u\le 0$ on $[\alpha, \beta]$ if
$$
0\le\lambda\le \frac{r\big(\frac r2-b\max(\alpha^\gamma,\beta^\gamma)\big)}{2\exp(r(\beta-\alpha))-1}.
$$
Choosing
$$
r=2+2b\max(\alpha^\gamma,\beta^\gamma),
$$
 it follows that the right hand side of the above inequality
is greater than
$\hat\lambda(\alpha,\beta)$.
We have thus shown that
there exists a positive function $u$ on $[\alpha,\beta]$ satisfying
$(L_{bx^\gamma}+\hat\lambda(\alpha,\beta))u\le0$ in $[\alpha,\beta]$ and $u'(\beta)\ge0$.
By the criticality theory of second order elliptic operators \cite[chapter 4]{P}, \cite{Pinc},
it follows that the principal eigenvalue for $-L_{bx^\gamma}$ on $(\alpha, \beta)$ with
the Dirichlet boundary condition at $\alpha$ and the Neumann boundary condition at $\beta$
is larger than $\hat\lambda(\alpha,\beta).$
By the Feynman-Kac  formula, when $\lambda$ is less than the aforementioned principal eigenvalue, the function
$u_\lambda(x)\equiv E_x^{bx^\gamma;\text{Ref}\leftarrow:\beta}\exp (\lambda T_\alpha)$ satisfies the boundary-value
problem
$(L_{bx^\gamma}+\lambda)u=0$ in $(\alpha,\beta)$, $u(\alpha)=1$ and $u'(\beta)=0$.
Since $\lambda$ is smaller than the principal eigenvalue, it follows from the generalized maximum principal \cite[chapter 3]{P}, \cite{Pinc} that
$u_\lambda\le u$, if
$u$ satisfies $(L+\lambda)u\le 0$ in $[\alpha,\beta]$, $u(\alpha)\ge1$ and $u'(\beta)\ge0$.
The calculation above showed that $u$
as defined in \eqref{u}, with $r=2+2b\max(\alpha^\gamma, \beta^\gamma)$,
satisfies these requirements; thus in particular,
\eqref{est-mgfalphabeta} holds.
\end{proof}
\begin{proposition}\label{2}
For $1\le x<\beta$,
\begin{equation}\label{prop2form}
E_x^{D;\text{\rm Ref}\rightarrow:1}\exp(\frac{D^2}2 T_\beta)=\frac{\exp(D(\beta-1))}{1+D(\beta-1)}\Big(1+D(x-1)\Big)\exp\big(-D(x-1)\big).
\end{equation}
\end{proposition}
\begin{proof}
The function
$$
u(x)=\frac{\exp(D(\beta-1))}{1+D(\beta-1)}\Big(1+D(x-1)\Big)\exp\big(-D(x-1)\big)
$$
solves the boundary value problem $(L_D+\frac{D^2}2)u=0$ in $(1,\beta)$ with
$u'(1)=0$ and $u(\beta)=1$.
Since $u>0$, it follows again from the criticality theory of elliptic operators that
the principal eigenvalue of $-L_D$ on $(1,\beta)$ with the Neumann boundary condition at 1
and the Dirichlet boundary condition at $\beta$ is greater than $\frac{D^2}2$.
Thus, $E_x^{D;\text{Ref}\rightarrow:1}\exp(\frac{D^2}2 T_\beta)<\infty$ and by the Feynman-Kac
formula, this function of $x\in[1,\beta]$ solves the above boundary value problem,
and consequently coincides with $u$.
\end{proof}
\begin{proposition}\label{3}
For $\lambda>0$ and $1<\alpha<\beta$,
\begin{equation*}
\begin{aligned}
&E_\beta^{D;\text{\rm Ref}\leftarrow:\beta}\exp(-\lambda T_\alpha)=\\
&\frac{2\sqrt{D^2+2\lambda}\thinspace e^{-2D(\beta-\alpha)}}
{(-D+\sqrt{D^2+2\lambda}\thinspace)\thinspace e^{(-D+\sqrt{D^2+2\lambda}\thinspace)(\beta-\alpha)}+
(D+\sqrt{D^2+2\lambda}\thinspace)\thinspace e^{(-D-\sqrt{D^2+2\lambda}\thinspace)(\beta-\alpha)}}.
\end{aligned}
\end{equation*}
\end{proposition}
\begin{proof}
By the Feynman-Kac formula, $E_x^{D;\text{Ref}\leftarrow:\beta}\exp(-\lambda T_\alpha)$, for $x\in[\alpha,\beta]$,
solves the boundary value problem
$(L_D-\lambda)u=0$ in $(\alpha, \beta)$, with $u(\alpha)=1$ and $u'(\beta)=0$.
The solution of this linear equation is given by
$$
u(x)=\frac{r_1e^{-r_1(\beta-\alpha)}e^{r_2(x-\alpha)}+r_2e^{r_2(\beta-\alpha)}e^{-r_1(x-\alpha)}}
{r_2e^{r_2(\beta-\alpha)}+r_1e^{-r_1(\beta-\alpha)}},
$$
where
$r_1=D+\sqrt{D^2+2\lambda}$ and $r_2=-D+\sqrt{D^2+2\lambda}$.
Substituting $x=\beta$ completes the proof.
\end{proof}

\begin{proposition}\label{4}
\begin{equation}
E_x^{bx^\gamma;\text{Ref}\rightarrow:\alpha}
\exp(\lambda\tau_\beta)\le 2,\ \text{for}\ x\in[\alpha,\beta]
\ \text{and} \ \lambda\le\bar\lambda,
\end{equation}
\end{proposition}
where $\bar \lambda=\frac{b\min(\alpha^\gamma,\beta^\gamma)}{(2e-1)(\beta-\alpha)}$.
\begin{proof}
The proof is similar to that of Proposition \ref{1}. By the Feynman-Kac formula, when
$\lambda$ is less than the principal eigenvalue for the operator
$L_{bx^\gamma}$ on $(\alpha,\beta)$ with the Neumann boundary condition at $\alpha$
and the Dirichlet boundary condition at $\beta$, the function $u_\lambda(x)\equiv E_x^{bx^\gamma;\text{Ref}\rightarrow:\alpha}
\exp(\lambda\tau_\beta)$ solves the equation $(L_{bx^\gamma}+\lambda)u=0$ in $(\alpha, \beta)$, $u'(\alpha)=0$ and
$u(\beta)=1$. Also, if $u>0$ satisfies $(L_{bx^\gamma}+\lambda)u\le 0$ in $(\alpha, \beta)$, $u'(\alpha)\le0$ and
$u(\beta)\ge1$, then $\lambda$ is smaller than the principal eigenvalue and $u_\lambda\le u$.
We look for such a function  $u$ in the form
$u(x)=2-\exp\big(-r(\beta-x)\big)$, where $r>0$.
Note then that  $u(\beta)=1$, $u'(\alpha)\le 0$ and $1\le u\le 2$ on $[\alpha,\beta]$.
We have
$$
\exp\big(r(\beta-x)\big)\big(L_{bx^\gamma}+\lambda\big)u=(-\frac12r^2-bx^\gamma r-\lambda)+2\lambda\exp\big(r(\beta-x)\big).
$$
It follows readily that if
\begin{equation}\label{lambdachoice}
\lambda\le \frac{\frac12r^2+br\min(\alpha^\gamma,\beta^\gamma)}{2\exp(r(\beta-\alpha))-1},
\end{equation}
then $(L_{bx^\gamma}+\lambda)u\le 0$ on $[\alpha,\beta]$.
With the choice $r=\frac1{\beta-\alpha}$ in \eqref{lambdachoice}, it is clear that $\bar\lambda$ in the statement of the proposition is smaller
than the right hand side of \eqref{lambdachoice}.  Thus, $u_\lambda(x)\le u(x)\le 2$, for $\lambda\le \bar\lambda$.
\end{proof}

\section{Proof of Theorem \ref{1}}
We will denote probabilities for the process staring from  1 at time 0
by  $P_1$.
Let $\mathcal{F}_t=\sigma(X(s), 0\le s\le t)$ denote the
 standard filtration
on real-valued continuous paths $X(t)$. By standard comparison results and the fact that
the  transience/recurrence dichotomy is not affected by a bounded change
in the drift over a compact set,
we may assume that
\begin{equation}\label{bfassump}
B(x)=bx^\gamma, \ \text{for all}\ x\ge1,\ \ \ \ \ \ \  f(t)=\begin{cases} 2,\ t\in[0,\exp\big((\frac2c)^{1+\gamma}\big)];\\
c(\log t)^{\frac1{1+\gamma}}, \ t> \exp\big((\frac2c)^{1+\gamma}\big).\end{cases}
\end{equation}
\medskip

\it\noindent Proof of (i).\rm\
Let $j_0=[(\frac2c)^{1+\gamma}]+1$. Let $t_j=e^j$. Then $f(t_j)=cj^{\frac1{1+\gamma}}$, for $j\ge j_0$.
For $j\ge j_0$, let $A_{j+1}$ denote the event that the process hits 1 at some time $t\in[t_j,t_{j+1}]$.
The conditional version of the Borel-Cantelli lemma \cite{D} shows that if
\begin{equation}\label{BC-cond-infty}
\sum_{j=j_0}^\infty P_1(A_{j+1}|\mathcal{F}_{t_j})=\infty, \ \text{a.s.},
\end{equation}
then  $P_1(A_j\ \text{i.o.})=1$,  and thus the process is recurrent.
Thus, to show recurrence, it suffices to show \eqref{BC-cond-infty}.

Since up to time $t_j$, the largest the process can be is $f(t_j)$, and
since
up to time $t_{j+1}$ the time-dependent domain is contained in $[1,f(t_{j+1})]$, it follows
by comparison that
\begin{equation}\label{Ajcond}
P_1(A_{j+1}|\mathcal{F}_{t_j})\ge P_{f(t_j)}^{bx^\gamma;\text{Ref}\leftarrow:f(t_{j+1})}(T_1\le t_{j+1}-t_j)\ \text{a.s.}
\end{equation}
We  estimate the right hand side of \eqref{Ajcond}.
Let $\sigma^{(j)}_0=0$, $\kappa^{(j)}_i=\inf\{t\ge\sigma^{(j)}_{i-1}: X(t)=f(t_{j+1})\}$ and
$\sigma^{(j)}_i=\inf\{t>\kappa^{(j)}_i: X(t)=f(t_j)\}$, $j\ge j_0,\ i=1, 2, \ldots$.
 For any $l_j\in \mathbb{N}$,
$$
\{T_1<\sigma^{(j)}_{l_j}\}-\{\sigma^{(j)}_{l_j}>t_{j+1}-t_j\}\subset
\{T_1\le t_{j+1}-t_j\}.
$$
Also,    it follows by the strong Markov property that
$$
P_{f(t_j)}^{bx^\gamma;\text{Ref}\leftarrow:f(t_{j+1})}(T_1<\sigma^{(j)}_{l_j})=1- \big(P_{f(t_j)}^{bx^\gamma;\text{Ref}\leftarrow:f(t_{j+1})}(T_{f(t_{j+1})}<T_1)\big)^{l_j}.
$$
Thus
 \begin{equation}\label{key}
\begin{aligned}
&P_{f(t_j)}^{bx^\gamma;\text{Ref}\leftarrow:f(t_{j+1})}(T_1\le t_{j+1}-t_j)\ge1- \big(P_{f(t_j)}^{bx^\gamma;\text{Ref}\leftarrow:f(t_{j+1})}(T_{f(t_{j+1})}<T_1)\big)^{l_j}-\\
&P_{f(t_j)}^{bx^\gamma;\text{Ref}\leftarrow:f(t_{j+1})}(\sigma^{(j)}_{l_j}>t_{j+1}-t_j).
\end{aligned}
\end{equation}
From \eqref{BC-cond-infty}-\eqref{key}, we will obtain $P_1(A_j\ \text{i.o.})=1$, and thus recurrence,  if we can select
$\{l_j\}_{j=1}^\infty$ such that
\begin{equation}\label{key1}
\sum_{j=j_0}^\infty\Big(1-\big(P_{f(t_j)}^{bx^\gamma;\text{Ref}\leftarrow:f(t_{j+1})}(T_{f(t_{j+1})}<T_1)\big)^{l_j}\Big)=\infty,
\end{equation}
and
\begin{equation}\label{key2}
\sum_{j=j_0}^\infty P_{f(t_j)}^{bx^\gamma;\text{Ref}\leftarrow:f(t_{j+1})}(\sigma^{(j)}_{l_j}>t_{j+1}-t_j)<\infty.
\end{equation}

Let
\begin{equation}\label{phi}
\phi(x)=
\int_x^\infty\exp(-\int_0^t2bs^\gamma ds)dt=\int_x^\infty \exp(-\frac{2bt^{1+\gamma}}{1+\gamma})dt, \ x\ge1.
\end{equation}
Since $L\phi=0$,
it follows by standard probabilistic potential theory \cite[chapter 5]{P} that
\begin{equation}\label{pottheory}
P_{f(t_j)}^{bx^\gamma;\text{Ref}\leftarrow:f(t_{j+1})}(T_{f(t_{j+1})}<T_1)=\frac{\phi(1)-\phi(f(t_j))}{\phi(1)-\phi(f(t_{j+1}))}=
1-\frac{\phi(f(t_j))-\phi(f(t_{j+1}))}{\phi(1)-\phi(f(t_{j+1}))}.
\end{equation}
Applying L'H\^opital's rule shows that
$$
\lim_{x\to\infty}\frac{\int_x^\infty\exp(-\frac{2bt^{1+\gamma}}{1+\gamma})dt}{x^{-\gamma}\exp(-\frac{2bx^{1+\gamma}}{1+\gamma})}
=\frac1{2b};
$$
thus,
\begin{equation}\label{phiasymp}
\phi(x)\sim\frac1{2b}x^{-\gamma}\exp(-\frac{2bx^{1+\gamma}}{1+\gamma}),\ \text{as}\ x\to\infty.
\end{equation}
Using the fact that $(1-t)^l\le \exp(-lt)\le1-lt+\frac12(lt)^2\le1-\frac12lt$, if $l,t\ge0$ and $lt\le 1$, along with
\eqref{pottheory}, we have
\begin{equation}\label{lfto0}
\begin{aligned}
&1-\big(P_{f(t_j)}^{bx^\gamma;\text{Ref}\leftarrow:f(t_{j+1})}(T_{f(t_{j+1})}<T_1)\big)^{l_j}
\ge\frac12l_j\frac{\phi(f(t_j))-\phi(f(t_{j+1}))}{\phi(1)-\phi(f(t_{j+1}))}, \\
& \text{for sufficiently large}\ j,\
 \text{if}\
\lim_{j\to\infty}l_j\phi(f(t_j))=0.
\end{aligned}
\end{equation}
Using \eqref{phiasymp} along with the facts that $f(x)=c(\log x)^{\frac1{1+\gamma}}$ and $t_j=e^j$,
it follows that there exists a $K_0\in(0,1)$ such that $\phi(f(t_{j+1}))\le K_0\phi(f(t_j))$ for all large
$j$. Thus,
\begin{equation}\label{asympphidiff}
\begin{aligned}
&\frac{\phi(f(t_j))-\phi(f(t_{j+1}))}{\phi(1)-\phi(f(t_{j+1}))}\ge K_1\phi(f(t_j))\ge K_2\thinspace j^{-\frac\gamma{1+\gamma}}
\exp(-\frac{2bc^{1+\gamma}}{1+\gamma}j),\\
& \text{for sufficiently large}\ j,
\end{aligned}
\end{equation}
for constants $K_1,K_2>0$.
From \eqref{lfto0} and \eqref{asympphidiff}, it follows that
\eqref{key1} will hold if we define $l_j\in\mathbb{N}$ by
\begin{equation}\label{ljselection}
l_j=[\frac1{j^{\frac1{1+\gamma}}\log j}
\exp(\frac{2bc^{1+\gamma}}{1+\gamma}j)],
\end{equation}
since then the general term,
$1-\big(P_{f(t_j)}^{bx^\gamma;\text{Ref}\leftarrow:f(t_{j+1})}(T_{f(t_{j+1})}<T_1)\big)^{l_j}$, in \eqref{key1} will be on the
order at least $\frac 1{j\log j}$.

With $l_j$ chosen as above,
we now analyze $P_{f(t_j)}^{bx^\gamma;\text{Ref}\leftarrow:f(t_{j+1})}(\sigma^{(j)}_{l_j}>t_{j+1}-t_j)$ and show that \eqref{key2} holds.
By the strong Markov property,
 $\sigma^{(j)}_{l_j}=\sum_{i=1}^{l_j}X_i+\sum_{i=1}^{l_j}Y_i$,
where $\{X_i\}_{i=1}^\infty$ is an IID sequence distributed according to $T_{f(t_{j+1})}$ under
$P_{f(t_j)}^{bx^\gamma;\text{Ref}\rightarrow:1}$, $\{Y_i\}_{i=1}^\infty$ is an IID sequence distributed according to $T_{f(t_j)}$ under
$P_{f(t_{j+1})}^{bx^\gamma;\text{Ref}\leftarrow:f(t_{j+1})}$, and the two IID sequences are independent of one another.
By Markov's inequality,
\begin{equation}\label{key2est}
\begin{aligned}
&P_{f(t_j)}^{bx^\gamma;\text{Ref}\leftarrow:f(t_{j+1})}(\sigma^{(j)}_{l_j}>t)\le\exp(-\lambda t)
E_{f(t_j)}^{bx^\gamma;\text{Ref}\leftarrow:f(t_{j+1})}\exp(\lambda\sigma^{(j)}_{l_j})=\\
&\exp(-\lambda t)\big(E_{f(t_j)}^{bx^\gamma;\text{Ref}\rightarrow:1}\exp(\lambda T_{f(t_{j+1})} )\big)^{l_j}
\big(E_{f(t_{j+1})}^{bx^\gamma;\text{Ref}\leftarrow:f(t_{j+1})}\exp(\lambda T_{f(t_j)})\big)^{l_j},
\end{aligned}
\end{equation}
for any $\lambda>0$.

By Proposition
\ref{1},
\begin{equation}\label{fromprop1}
E_{f(t_{j+1})}^{bx^\gamma;\text{Ref}\leftarrow:f(t_{j+1})}\exp(\lambda T_{f(t_j)})\le
 2,\ \text{for}\ \lambda\le \hat \lambda(f(t_j),f(t_{j+1})),
\end{equation}
where $\hat\lambda(\cdot,\cdot)$ is as in \eqref{hatlambda}.
Using the fact that $f(t_j)=cj^\frac1{1+\gamma}$, it is  easy to check that
there exists a $\hat\lambda_0>0$ such that
\begin{equation}\label{hatlambdabound}
 \hat \lambda(f(t_j),f(t_{j+1}))\ge\hat\lambda_0,\ \text{for all}\ j\ge j_0.
\end{equation}

By comparison,
\begin{equation}\label{comparisonprop2}
E_{f(t_j)}^{bx^\gamma;\text{Ref}\rightarrow:1}\exp(\lambda T_{f(t_{j+1})} )\le
E_{f(t_j)}^{D_j;\text{Ref}\rightarrow:1}\exp(\lambda T_{f(t_{j+1})}),
\end{equation}
if
$$
D_j\le\min_{x\in[1,f(t_{j+1})]}bx^\gamma.
$$
If $\gamma\ge0$, choose $D_j=\min(b,\sqrt{2\hat\lambda_0}\thinspace)$, for all $j\ge j_0$; thus,
$\frac{D_j^2}2\le\hat\lambda_0$. If $\gamma\in(-1,0)$, choose
$D_j=b(f(t_{j+1}))^\gamma=bc^\gamma(j+1)^{\frac\gamma{1+\gamma}}$.
With these choices of $D_j$, we have for all $\gamma>-1$,
\begin{equation}\label{Dhatlambda}
\frac{D_j^2}2\le \hat\lambda_0,\ \text{ for sufficiently large}\ j.
\end{equation}
It is easy to check that if one substitutes $D=D_j$, $x=f(t_j)=c(\log j)^{\frac1{1+\gamma}}$ and $\beta=f(t_{j+1})=c(\log (j+1))^\frac1{1+\gamma}$ in the expression on the right hand side of \eqref{prop2form} in Proposition \ref{2}, the resulting expression
is bounded   in $j$. Letting $M>1$ be  an upper bound, it follows that
\begin{equation}\label{Mbound}
E_{f(t_j)}^{D_j;\text{Ref}\rightarrow:1}\exp(\frac{D_j^2}2 T_{f(t_{j+1})} )\le M.
\end{equation}
Noting that $t_{j+1}-t_j=e^{j+1}-e^j\ge e^j$, and
choosing $\lambda=\frac{D_j^2}2$ in \eqref{key2est}, it follows from \eqref{key2est}-\eqref{Mbound}
that
\begin{equation}\label{key3est}
P_{f(t_j)}^{bx^\gamma;\text{Ref}\leftarrow:f(t_{j+1})}(\sigma^{(j)}_{l_j}>t_{j+1}-t_j)\le\exp(-\frac{D_j^2}2e^j)(2M)^{l_j},\ \text{for sufficiently large}\ j.
\end{equation}
Recalling $l_j$ from
\eqref{ljselection}, we conclude from \eqref{key3est} that
\begin{equation}\label{finalsigmaj}
\begin{aligned}
&P_{f(t_j)}^{bx^\gamma;\text{Ref}\leftarrow:f(t_{j+1})}
(\sigma^{(j)}_{l_j}>t_{j+1}-t_j)\le \exp(-\frac{D_j^2}2e^j)(2M)^{j^{-\frac1{1+\gamma}}(\log j)^{-1}
\exp(\frac{2bc^{1+\gamma}}{1+\gamma}j)}=\\
& \exp(-\frac{D_j^2}2e^j)\exp\Big(j^{-\frac1{1+\gamma}}(\log j)^{-1}
e^{\frac{2bc^{1+\gamma}}{1+\gamma}j}\log 2M\Big),\
\text{for sufficiently large}\ j.
\end{aligned}
\end{equation}
Recalling that $D_j$ is equal to a positive constant, if $\gamma\ge0$,  and that $D_j$
is on the order $j^\frac{\gamma}{1+\gamma}$, if $\gamma<0$, it follows
that the right hand side of \eqref{finalsigmaj} is summable in $j$ if $\frac{2bc^{1+\gamma}}{1+\gamma}<1$,
or if
$\frac{2bc^{1+\gamma}}{1+\gamma}=1$ and $\gamma\ge-\frac12$. Thus \eqref{key2} holds for this range
of $b,c$ and $\gamma$.
This completes the proof of (i).
\medskip

\noindent \it Proof of (ii).\rm\
Let $j_1=[\exp\big((\frac2c)^{1+\gamma}\big)]+1$. Then $f(j)=c(\log j)^\frac1{1+\gamma}$, for $j\ge j_1$.
For $j\ge j_1$, let $B_j$ be the event that the process hits 1 sometime between the first time it hits
$f(j)$ and the first time it hits $f(j+1)$: $B_j=\{X(t)=1\
\text{for some}\ t\in(T_{f(j)}, T_{f(j+1)})\}$.
If we show that
\begin{equation}\label{BC-finite}
\sum_{j=j_1}^\infty P_1(B_j)<\infty,
\end{equation}
then by the Borel-Cantelli lemma it will follow that $P_1(B_j\ \text{i.o.})=0$, and consequently the process
is transient.

To prove \eqref{BC-finite}, we need to use different methods depending on whether $\gamma\le0$ or $\gamma>0$.
We begin with the case $\gamma\le 0$.
To consider whether or not the event $B_j$ occurs, we first wait until time $T_{f(j)}$.
Of course, necessarily, $T_{f(j)}\ge j$, since $f(j)$ is not accessible to the process before time $j$.
Since we may have $T_{f(j)}<j+1$, the point $f(j+1)$ may not be accessible to the process
at time $T_{f(j)}$, however, if we wait one unit of time, then after that, the point $f(j+1)$ certainly
will be accessible, since $T_{f(j)}+1\ge j+1$. Let $M_j<f(j)-1$. Now if in that one unit of time, the process
never got to the level $f(j)-M_j$, then by comparison, the probability of $B_j$ occurring is no more
than $P_{f(j)-M_j}^{bx^\gamma;\text{Ref}\leftarrow:f(j+1)}(T_1<T_{f(j+1)})$
(because after this one unit of time the process will be at a position greater than or equal to
$f(j)-M_j$). By comparison with the process that is reflected at the fixed point $f(j)$,
the probability that the process got to the level
$f(j)-M_j$ in that one unit of time is bounded from above
by $P_{f(j)}^{bx^\gamma;\text{Ref}\leftarrow:f(j)}(T_{f(j)-M_j}\le 1)$.
From these considerations, we conclude that
\begin{equation}\label{estBj}
P_1(B_j)\le P_{f(j)-M_j}^{bx^\gamma;\text{Ref}\leftarrow:f(j+1)}(T_1<T_{f(j+1)})+P_{f(j)}^{bx^\gamma;\text{Ref}\leftarrow:f(j)}(T_{f(j)-M_j}\le 1).
\end{equation}

Similar to \eqref{pottheory}, we have
\begin{equation}\label{pottheory2}
P_{f(j)-M_j}^{bx^\gamma;\text{Ref}\leftarrow:f(j+1)}(T_1<T_{f(j+1)})=\frac{\phi(f(j)-M_j)-\phi(f(j+1))}{\phi(1)-\phi(f(j+1))}.
\end{equation}
For  $\epsilon\in(0,1)$ to be chosen later sufficiently small, choose $M_j=\epsilon f(j)$. Recall that
$f(j)=c(\log j)^{\frac1{1+\gamma}}$. Then from
\eqref{phiasymp}
we have
\begin{equation}\label{estimatef-M}
\begin{aligned}
&\phi(f(j)-M_j)=\phi\big(c(1-\epsilon)(\log j)^\frac1{1+\gamma}\big)\sim\\
&\frac1{2b}\big(c(1-\epsilon)(\log j)^{\frac1{1+\gamma}}  \big)^{-\gamma}\exp\big(-\frac{2b(c(1-\epsilon))^{1+\gamma}\log j}{1+\gamma}\big)=\\
&\frac1{2b}\big(c(1-\epsilon)(\log j)^{\frac1{1+\gamma}}  \big)^{-\gamma}\thinspace j^{-\frac{2b(c(1-\epsilon))^{1+\gamma}}{1+\gamma}}.
\end{aligned}
\end{equation}
Since by assumption, $\frac{2bc^{1+\gamma}}{1+\gamma}>1$, we can select $\epsilon\in(0,1)$ such that
$\frac{2b(c(1-\epsilon))^{1+\gamma}}{1+\gamma}>1$. With such a choice of $\epsilon$, it follows from \eqref{pottheory2}
and \eqref{estimatef-M} that
\begin{equation}\label{firstguy}
\sum_{j=j_1}^\infty P_{f(j)-M_j}^{bx^\gamma;\text{Ref}\leftarrow:f(j+1)}(T_1<T_{f(j+1)})<\infty.
\end{equation}

We now estimate $P_{f(j)}^{bx^\gamma;\text{Ref}\leftarrow:f(j)}(T_{f(j)-M_j}\le 1)$, where $M_j=\epsilon f(j)$,
with $\epsilon$  as above.
By comparison, we have
\begin{equation}\label{compareconstantdrift}
P_{f(j)}^{bx^\gamma;\text{Ref}\leftarrow:f(j)}(T_{f(j)-M_j}\le 1)\le P_{f(j)}^{D_j;\text{Ref}\leftarrow,f(j)}(T_{f(j)-M_j}\le 1),
\end{equation}
where $D_j$ is  equal to  the minimum of the original drift on the interval
$[f(j)-M_j,f(j)]$; that is,
$$
D_j= bc^\gamma(\log j)^\frac\gamma{1+\gamma}.
$$
By Markov's inequality, we have for $\lambda>0$,
\begin{equation}\label{Markagain}
P_{f(j)}^{D_j;\text{Ref}\leftarrow,f(j)}(T_{f(j)-M_j}\le 1)\le \exp(\lambda)E_{f(j)}^{D_j;\text{Ref}\leftarrow,f(j)}\exp(-\lambda T_{f(j)-M_j}).
\end{equation}
Using Proposition \ref{3} with $\alpha=f(j)-M_j$, $\beta=f(j)$ and $D=D_j$, we have
\begin{equation}\label{fromprop3}
\begin{aligned}
&E_{f(j)}^{D_j;\text{Ref}\leftarrow:f(j)}\exp(-\lambda T_{f(j)-M_j})=\\
&\frac{2\sqrt{D_j^2+2\lambda}\thinspace e^{-2D_jM_j}}
{(-D_j+\sqrt{D_j^2+2\lambda}\thinspace)\thinspace e^{(-D_j+\sqrt{D_j^2+2\lambda}\thinspace)M_j}+
(D_j+\sqrt{D_j^2+2\lambda}\thinspace)\thinspace e^{(-D_j-\sqrt{D_j^2+2\lambda}\thinspace)M_j}}.
\end{aligned}
\end{equation}

If $\gamma<0$, then $\lim_{j\to\infty}D_j=0$ and $M_j\to\infty$, and
it follows  from \eqref{fromprop3} that
\begin{equation}\label{gamma<0}
E_{f(j)}^{D_j;\text{Ref}\leftarrow:f(j)}\exp(-\lambda T_{f(j)-M_j})
\le K\exp(-\sqrt{2\lambda}\thinspace M_j),
\end{equation}
for some $K>0$.
If $\gamma=0$, then $D_j=b$, for all $j$, and we have from \eqref{fromprop3},
\begin{equation}\label{gamma=0}
\begin{aligned}
&E_{f(j)}^{D_j;\text{Ref}\leftarrow:f(j)}\exp(-\lambda T_{f(j)-M_j})\sim\frac{2\sqrt{b^2+2\lambda}}{-b+\sqrt{b^2+2\lambda}}\exp\big(-(b+(\sqrt{b^2+2\lambda}\thinspace)M_j\big),\\
&\text{as}\ j\to\infty.
\end{aligned}
\end{equation}
Since $M_j=\epsilon c(\log j)^{\frac1{1+\gamma}}$, it follows from \eqref{gamma<0} and \eqref{gamma=0} that
\begin{equation}\label{gammanonpos}
\sum_{j=j_1}^\infty E_{f(j)}^{D_j;\text{Ref}\leftarrow:f(j)}\exp(-\lambda T_{f(j)-M_j})<\infty,
\end{equation}
for all choices of $\lambda>0$ in the case $\gamma<0$, and   for sufficiently large $\lambda$
in the case  $\gamma=0$.
Thus, we conclude  from \eqref{gammanonpos} and  \eqref{Markagain} that
\begin{equation}\label{finalgneg}
\sum_{j=j_1}^\infty P_{f(j)}^{D_j;\text{Ref}\leftarrow,f(j)}(T_{f(j)-M_j}\le 1)<\infty.
\end{equation}
Now  \eqref{BC-finite} follows from \eqref{estBj}, \eqref{firstguy} and \eqref{finalgneg}.
\medskip

We now turn to the case that $\gamma>0$.
Let $\zeta_{j+1}=\inf\{t\ge j+1: X(t)\ge f(j)\}$.
Since the process cannot reach $f(j+1)$ before time $j+1$, it follows that
$T_{f(j)}\le \zeta_{j+1}\le T_{f(j+1)}$.
Let $C_j=\{X(t)=1 \ \text{for some}\ t\in(T_{f(j)},\zeta_{j+1})\}$,
and let $G_j=\{X(t)=1 \ \text{for some}\ t\in (\zeta_{j+1},T_{f(j+1)})\}$.
Then $B_j=C_j\cup G_j$; thus,
\begin{equation}\label{BCG}
P_1(B_j)\le P_1(C_j)+P_1(G_j).
\end{equation}

Since the right hand endpoint of the domain is larger than or equal to $f(t_{j+1})$
at all times $t\ge\zeta_{j+1}$, it follows by comparison that
$P_1(G_j)\le P_{f(j)}^{bx^\gamma;\text{Ref}\leftarrow:f(j+1)}(T_1<T_{f(j+1)})$.
Thus, similar to \eqref{pottheory} we have
\begin{equation}\label{Gestimate}
P_1(G_j)\le\frac{\phi(f(j))-\phi(f(j+1))}{\phi(1)-\phi(f(j+1))}.
\end{equation}
As in \eqref{estimatef-M}, but with $\epsilon=0$, we have
\begin{equation}\label{estimatef}
\phi(f(j))\sim\frac1{2b}\big(c(\log j)^{\frac1{1+\gamma}}  \big)^{-\gamma}\thinspace j^{-\frac{2bc^{1+\gamma}}{1+\gamma}}.
\end{equation}
From \eqref{Gestimate}, \eqref{estimatef} and the fact that $\frac{2bc^{1+\gamma}}{1+\gamma}>1$, it follows that
\begin{equation}\label{Gjconverges}
\sum_{j=j_1}^\infty P_1(G_j)<\infty.
\end{equation}

For any $s_j$, we have the estimate
\begin{equation}\label{Cest}
P_1(C_j)\le P_{f(j)}^{bx^\gamma;\text{Ref}\leftarrow:f(j)}(T_1\le s_j+1)+P_1^{b;\text{Ref}\rightarrow:1}(T_{f(j)}>s_j).
\end{equation}
Here is the explanation for the above estimate.
To check whether or not the event $C_j$ occurs, one waits until time $T_{f(j)}$, at which time
the process has first reached $f(j)$. Of course $T_{f(j)}\ge j$. If in fact,
$T_{f(j)}\ge j+1$, then $\zeta_{j+1}=T_{f(j)}$ and $C_j$ does not occur.
Otherwise, one watches the process between time $T_{f(j)}$ and time $j+1$.
If the process hit 1 in this time interval, whose length is no more than $1$,
 then $C_j$ occurs.
(Note that during this interval of time, the right hand boundary for reflection  is always
at least $f(j)$.)
 Otherwise, $C_j$ has not yet occurred,
 but one continues to watch
the process after time $j+1$ until the first time the process is again greater than or equal
to $f(j)$.
If the process reaches 1 in this interval, then $C_j$ occurs, while if not, then
we conclude that $C_j$ did not occur.  (Note that if $X(j+1)\ge f(j)$, then the length of this
final time interval is 0.) The random variable denoting the length of this final time interval is stochastically
dominated by the random variable $T_{f(j)}$ under $P_1^{b;\text{Ref}\rightarrow:1}$, since
the actually drift is always larger than or equal to $b$ everywhere, and the actual starting point
of the process at the beginning of this final time interval is certainly greater than or equal to 1.
In the estimate \eqref{Cest}, one should think of $s_j$ as a possible value for the length of this final
time interval.

We first estimate
$P_1^{b;\text{Ref}\rightarrow:1}(T_{f(j)}>s_j)$, the second term on the right hand side of \eqref{Cest}.
By Markov's inequality, for any $\lambda>0$,
\begin{equation}\label{Markin}
P_1^{b;\text{Ref}\rightarrow:1}(T_{f(j)}>s_j)\le
\exp(-\lambda s_j)E_1^{b;\text{Ref}\rightarrow:1}\exp(\lambda T_{f(j)}).
\end{equation}
Applying Proposition 2 with $D=b$, $x=1$ and $\beta=f(j)=c(\log j)^{\frac1{1+\gamma}}$,
we have
\begin{equation}\label{anothermarkov}
E_1^{b;\text{Ref}\rightarrow:1}\exp(\frac{b^2}2 T_{f(j)})=
\frac{\exp\Big(b\big(c(\log j)^{\frac1{1+\gamma}}-1\big)\Big)}{1+b(c(\log j)^{\frac1{1+\gamma}}-1)}.
\end{equation}
Letting
\begin{equation}\label{sj}
s_j=\frac4{b^2}\log j,
\end{equation}
it follows from \eqref{Markin} with $\lambda=\frac{b^2}2$, \eqref{anothermarkov} and the fact that $\gamma>0$ that
\begin{equation}\label{finalsecondterm}
\sum_{j=j_1}^\infty P_1^{b;\text{Ref}\rightarrow:1}(T_{f(j)}>s_j)<\infty.
\end{equation}

We now estimate  $P_{f(j)}^{bx^\gamma;\text{Ref}\leftarrow:f(j)}(T_1\le s_j+1)$, the first term on the right hand side
 of \eqref{Cest}, where  $s_j$ has now been defined in \eqref{sj}.
Note that by the strong Markov property,
$T_1=T_{[f(t_j)]}+\sum_{i=2}^{[f(t_j)]} (T_i-T_{i-1})$, where
$\{T_i-T_{i-1}\}_{i=2}^{[f(t_j)]}$ and $T_{[f(t_j)]}$ are independent random variables
under $P_{f(j)}^{bx^\gamma;\text{Ref}\leftarrow:f(j)}$, and
 $T_i-T_{i-1}$ is distributed as $T_{i-1}$ under $P_i^{bx^\gamma;\text{Ref}\leftarrow:f(j)}$.
Let $\{X_i\}_{i=2}^{[f(j)]}$ be independent random variables with
$X_i$ distributed as $T_1$ under $P_2^{D_i;\text{Ref}\leftarrow:2}$,
where
\begin{equation}\label{Diformula}
D_i=b(i-1)^\gamma.
\end{equation}
We will use the generic $P$ and $E$ for calculating probabilities and expectations for the $X_i$.
Note that $D_i$ is the minimum of the original drift on the interval $[i-1,i]$.
Also note that  when one considers $T_{i-1}$ under $P_i^{bx^\gamma;\text{Ref}\leftarrow:f(j)}$, the process
gets reflected at $f(j)$, which is to the right of the starting point $i$, while when one considers
$T_1$ under $P_2^{D_i;\text{Ref}\leftarrow:2}$, the process gets reflected at its starting point.
Thus, by comparison, it follows that the distribution of $T_i-T_{i-1}$ under
$P_i^{bx^\gamma;\text{Ref}\leftarrow:f(j)}$
dominates the distribution of $X_i$, and consequently, the distribution of
$T_1$ under $P_{f(j)}^{bx^\gamma;\text{Ref}\leftarrow:f(j)}$ dominates the distribution of
$\sum_{i=2}^{[f(j)]} X_i$.
Thus,  we have
\begin{equation}\label{dominateiid}
 P_{f(j)}^{bx^\gamma;\text{Ref}\leftarrow:f(j)}(T_1\le s_j+1)\le  P(\sum_{i=2}^{[f(j)]}X_i\le s_j+1).
\end{equation}

By Markov's inequality, we have for any $\lambda>0$,
\begin{equation}\label{Markovineq}
\begin{aligned}
&P(\sum_{i=2}^{[f(j)]}X_i\le s_j+1)\le \exp(\lambda (s_j+1))E\exp(-\lambda\sum_{i=2}^{[f(j)]}X_i)
=\\
& \exp(\lambda (s_j+1))\prod_{i=2}^{[f(j)]}E_2^{D_i;\text{Ref}\leftarrow:2}\exp(-\lambda T_1).
\end{aligned}
\end{equation}
Applying Proposition \ref{3} with $\alpha=1$, $\beta=2$ and $D=D_i$,
we have
\begin{equation}\label{appprop3}
\begin{aligned}
&E_2^{D_i;\text{Ref}\leftarrow:2}\exp(-\lambda T_1)=\\
&\frac{2\sqrt{D_i^2+2\lambda}\thinspace e^{-2D_i}}
{(-D_i+\sqrt{D_i^2+2\lambda}\thinspace)\thinspace e^{(-D_i+\sqrt{D_i^2+2\lambda}\thinspace)}+
(D_i+\sqrt{D_i^2+2\lambda}\thinspace)\thinspace e^{(-D_i-\sqrt{D_i^2+2\lambda}\thinspace)}}.
\end{aligned}
\end{equation}
For fixed $\lambda>0$,
$-D_i+\sqrt{D_i^2+2\lambda}\sim\frac\lambda{D_i}$, as $D_i\to\infty$. Thus,
\eqref{appprop3} yields
\begin{equation}\label{asymexp}
E_2^{D_i;\text{Ref}\leftarrow:2}\exp(-\lambda T_1)\sim\frac{2D_i^2}\lambda \exp(-2D_i),\ \text{as}\ D_i\to\infty.
\end{equation}
From \eqref{Diformula} and \eqref{asymexp}, it follows that there exists a $K_0>0$ such that
\begin{equation}\label{productmgfs}
\begin{aligned}
&\prod_{i=2}^{[f(j)]}E_2^{D_i;\text{Ref}\leftarrow:2}\exp(-\lambda T_1)
\le\prod_{i=2}^{[f(j)]}\frac{2D_i^2K_0}\lambda \exp(-2D_i)=\\
&\prod_{i=1}^{[f(j)]-1}\frac{2K_0b^2i^{2\gamma}}\lambda\exp(-2bi^\gamma).
\end{aligned}
\end{equation}

We have
\begin{equation}\label{producti2gamma}
\prod_{i=1}^{[f(j)]-1}i^{2\gamma}\le (f(j))^{2\gamma f(j)}=\big(c(\log j)^\frac1{1+\gamma}\big)^{2\gamma c(\log j)^\frac1{1+\gamma}}.
\end{equation}
Also, for some $C_\gamma>0$,
$$
\sum_{i=1}^{[f(j)]-1}i^\gamma\ge \frac{(f(j))^{1+\gamma}}{1+\gamma}-C_\gamma(f(j))^\gamma
=\frac{c^{1+\gamma}\log j}{1+\gamma}-C_\gamma c^\gamma(\log j)^\frac{\gamma}{1+\gamma};
$$
thus,
\begin{equation}\label{sumigamma}
\prod_{i=1}^{[f(j)]-1}\exp(-2bi^\gamma)\le \exp\big(2bC_\gamma c^\gamma(\log j)^\frac{\gamma}{1+\gamma}\big)\thinspace j^{-\frac{2bc^{1+\gamma}}{1+\gamma}}.
\end{equation}
Then
from  \eqref{dominateiid}, \eqref{Markovineq},  and \eqref{productmgfs}-\eqref{sumigamma}, we have
\begin{equation}\label{firstterm}
\begin{aligned}
&P_{f(j)}^{bx^\gamma;\text{Ref}\leftarrow:f(j)}(T_1\le s_j+1)\le
\exp(\lambda (s_j+1))\times\\
&(1\vee\frac{2K_0b^2}\lambda)^{c(\log j)^{\frac1{1+\gamma}}}\big(c(\log j)^\frac1{1+\gamma}\big)^{2\gamma c(\log j)^\frac1{1+\gamma}}
 \exp\big(2bC_\gamma c^\gamma(\log j)^\frac{\gamma}{1+\gamma}\big)\thinspace j^{-\frac{2bc^{1+\gamma}}{1+\gamma}}.
\end{aligned}
\end{equation}
From \eqref{sj}, $s_j=\frac4{b^2}\log j$; so
$\exp(\lambda (s_j+1))=e^\lambda j^{\frac{4\lambda}{b^2}}$.
By assumption, $\frac{2bc^{1+\gamma}}{1+\gamma}>1$.
Thus, choosing $\lambda>0$ sufficiently small so that $ \frac4{b^2}\lambda-\frac{2bc^{1+\gamma}}{1+\gamma}<-1$,
and recalling that $\gamma>0$, it follows from \eqref{firstterm} that
\begin{equation}\label{finalfirstterm}
\sum_{j=j_1}^\infty P_{f(j)}^{bx^\gamma;\text{Ref}\leftarrow:f(j)}(T_1\le s_j+1)<\infty.
\end{equation}
(To see this easily, it is useful to convert the long expression on the right hand side of \eqref{firstterm} to exponential
form, similar to what was done in the equality in \eqref{finalsigmaj}.)
From  \eqref{Cest},  \eqref{finalsecondterm} and  \eqref{finalfirstterm} we conclude that
\begin{equation}\label{Cjconverges}
\sum_{j=j_1}^\infty P_1(C_j)<\infty.
\end{equation}
Now \eqref{BCG}, \eqref{Gjconverges} and \eqref{Cjconverges} give
\eqref{BC-finite} and complete the proof of the theorem.
\hfill $\square$

\section{Proof of Theorem \ref{2}}

First we prove Theorem \ref{2} in the case that $\mathcal{K}$ is a ball.
The part of the operator $\frac12\Delta+b\cdot\nabla$ involving radial derivatives is $\frac12\frac{d^2}{dr^2}+(\frac{d-1}{2r}+
b(x)\cdot\frac x{|x|})\frac d{dr}$.
Of course, in general, $b(x)\cdot\frac x{|x|}$ depends not only  on the radial component $r=|x|$ of $x$, but also on
the spherical component $\frac x{|x|}$.
Let
$B^+(r)=\max_{|x|=r}b(x)\cdot\frac x{|x|}$ and
$B^-(r)=\min_{|x|=r}b(x)\cdot\frac x{|x|}$. Then by comparison, if the multi-dimensional process
with radial drift $B^+(|x|)\cdot\frac x{|x|}$ is recurrent,  so is the one with drift $b(x)$, and if
the multi-dimensional process with radial drift $B^-(|x|)\cdot\frac x{|x|}$ is transient,  so is the one with drift $b(x)$.
In the case of a  radial drift  $B(|x|)\cdot\frac x{|x|}$, with
 $\mathcal{K}$  a ball, so that $D_t=f(t)\mathcal{K}$ is a ball, the question of transience/recurrence is equivalent to the question of transience/recurrence considered
in Theorem \ref{1} with drift $B(x)+\frac{d-1}{2x}$ and with $D_t=\big(1,\text{rad}(\mathcal{K})\thinspace f(t)\big)$,
where $\text{rad}(\mathcal{K})$ is the radius of $\mathcal{K}$.
 Thus, if $B(r)\equiv B^+(r)$ and $f(t)$ satisfy the inequalities  \eqref{recurcondition}
in part (i) of Theorem \ref{2} with $\frac{2bc^{1+\gamma}}{1+\gamma}<1$, then the
multi-dimensional process is recurrent, while if
 $B(r)\equiv B^-(r)$ and $f(t)$ satisfy the inequalities
\eqref{transcondition} in part (ii) of Theorem \ref{2} with $\frac{2bc^{1+\gamma}}{1+\gamma}>1$, then the
multi-dimensional process is transient.
(Of course, since $\mathcal{K}$ is a ball, $\text{rad}^\pm(\mathcal{K})$ appearing in Theorem \ref{1}
are equal to $\text{rad}(\mathcal{K})$.)

Now consider the case that $B(r)\equiv B^+(r)$ and $f(t)$ satisfy the inequalities
\eqref{recurcondition} in part (i) of Theorem \ref{2} with $\frac{2bc^{1+\gamma}}{1+\gamma}=1$.
To show recurrence, we need to show recurrence for the one dimensional case
when $B(x)=b x^\gamma+\frac{d-1}{2x}$, for large $x$, and $f(t)=c(\log t)^{\frac1{1+\gamma}}$,
for large $t$, with $\frac{2bc^{1+\gamma}}{1+\gamma}=1$.
Thus, the function $\phi$ appearing in \eqref{phi} must be replaced by
$$
\phi(x)=\int_x^\infty\exp(-\int_1^t(2bs^\gamma+\frac{d-1}s)ds)=
C\int_x^\infty t^{1-d}\exp(-\frac{2bt^{1+\gamma}}{1+\gamma})dt.
$$
(Here $C$ is the appropriate constant. In \eqref{phi} we integrated over $s$ starting from 0 for convenience in order
to prevent such a constant
from entering, however in the present case we can't do this because of the term $\frac{d-1}s$.)
In place of \eqref{phiasymp}, we will now have
$$
\phi(x)\sim \frac C{2b}x^{-\gamma+1-d}\exp(-\frac{2bx^{1+\gamma}}{1+\gamma}).
$$
This causes the term $j^{-\frac\gamma{1+\gamma}}$ on the right hand side of
\eqref{asympphidiff} to be replaced by
$j^{-\frac{\gamma+d-1}{1+\gamma}}$, which in turn  causes $l_j$ in \eqref{ljselection}  to be changed
to $l_j= [\frac{j^{\frac{d-2}{1+\gamma}}}{\log j}
\exp(\frac{2bc^{1+\gamma}}{1+\gamma}j)]$.
Finally, this causes the term
on the right hand side of \eqref{finalsigmaj} to be changed to
$\exp(-\frac{D_j^2}2e^j)\exp\Big(j^{\frac{d-2}{1+\gamma}}(\log j)^{-1}
e^{\frac{2bc^{1+\gamma}}{1+\gamma}j}\log 2M\Big)$.
Recalling that $D_j$ is equal to a positive constant, if $\gamma\ge0$, and $D_j$ is on the order
$j^{\frac\gamma{1+\gamma}}$, if $\gamma<0$, we conclude that if $\frac{2bc^{1+\gamma}}{1+\gamma}=1$, then
the above expression is summable in $j$ if
$d=2$ and $\gamma\ge0$. This proves recurrence when
$\frac{2bc^{1+\gamma}}{1+\gamma}=1$, $d=2$ and $\gamma\ge0$.

We now extend from the radial case to the case of general $\mathcal{K}$.
In \cite{DHS}, the proof of a condition for transience was first given for the radial case.
The extension to the case of general $\mathcal{K}$, which appears as step III in the proof
of Theorem 1.15 in that paper,  followed by
 Lemma 2.1 in that paper. This lemma
 implies that if  one considers two such processes, one  corresponding to $\mathcal{K}_1$ and one corresponding to $\mathcal{K}_2$, where
$\mathcal{K}_1$ is a ball and $\mathcal{K}_2\supset \bar{\mathcal{K}_1}$, then the process
corresponding to $\mathcal{K}_2$ is transient if the one
corresponding to $\mathcal{K}_1$ is transient.
Lemma 2.1 goes through just as well when the Brownian motion is replaced by our Brownian motion with
drift. This extends our proof of  transience to the case of general $\mathcal{K}$.

In \cite{DHS}, the proof of the condition for recurrence also was  first given in the radial case.
The extension to the general case, which is more involved than in the case of transience, and
which requires the additional condition $\int_0^\infty (f')^2(t)dt<\infty$,
appears in step V in the proof of Theorem 1.15 in that paper.
The analysis in that step also go through when Brownian motion is replaced by our Brownian motion with
drift. This extends the proof of recurrence to the case of general $\mathcal{K}$.

\hfill $\square$

\section{Proof of Theorem \ref{3}}
We will prove the theorem for the one-dimensional case. The proof for the multi-dimensional case
follows from the proof of the one-dimensional case, similar to the way the proof of Theorem \ref{2} follows
from the proof of Theorem \ref{1}.
Let $P_2$ and $E_2$ denote probabilities and expectations for the process starting from $x=2$ at time 0.

Let $t_j=e^j$ as in the proof of part (i) of Theorem \ref{1}.
We have
\begin{equation}\label{expupperbd}
E_2T_1\le t_1+\sum_{j=1}^\infty t_{j+1}P_2(T_1\ge t_j)=e+\sum_{j=1}^\infty e^{j+1}P_2(T_1\ge t_j).
\end{equation}
Recall the definition of $j_0$ and of $A_{j+1}$ from the beginning of the proof of part (i)
of Theorem \ref{1}.  From \eqref{Ajcond} we have for $j\ge j_0+1$,
\begin{equation}\label{intersectA}
P_2(T_1\ge t_j)\le P_2(\cap_{i=j_0}^{j-1}A_{i+1}^c)\le\prod_{i=j_0}^{j-1}\Big( 1-P_{f(t_i)}^{bx^\gamma;\text{Ref}\leftarrow:f(t_{i+1})}(T_1\le t_{i+1}-t_i)\Big).
\end{equation}
If we show that
\begin{equation}\label{goto0}
\lim_{j\to\infty}P_{f(t_j)}^{bx^\gamma;\text{Ref}\leftarrow:f(t_{j+1})}(T_1\le t_{j+1}-t_j)=1,
\end{equation}
then it will certainly follow from
\eqref{expupperbd} and \eqref{intersectA}  that
$E_2T_1<\infty$, proving positive recurrence.
In order to prove \eqref{goto0}, it suffices  from  \eqref{key} to prove that for some choice of positive integers $\{l_j\}_{j=j_0}^\infty$,
\begin{equation}\label{keyposrec1}
\lim_{j\to\infty}\big(P_{f(t_j)}^{bx^\gamma;\text{Ref}\leftarrow:f(t_{j+1})}(T_{f(t_{j+1})}<T_1)\big)^{l_j}=0
\end{equation}
and
\begin{equation}\label{keyposrec2}
\lim_{j\to\infty}P_{f(t_j)}^{bx^\gamma;\text{Ref}\leftarrow:f(t_{j+1})}(\sigma^{(j)}_{l_j}>t_{j+1}-t_j)=0.
\end{equation}
From  \eqref{pottheory}, \eqref{asympphidiff} and the fact that $\lim_{y\to\infty}(1-\frac1y)^{yg(y)}=0$, if
$\lim_{y\to\infty}g(y)=\infty$, it follows that
\eqref{keyposrec1} holds if
we choose
\begin{equation}\label{ljagain}
l_j=[j^{\frac\gamma{1+\gamma}}(\log j)\exp(\frac{2bc^{1+\gamma}}{1+\gamma}j)].
\end{equation}
With this choice of $l_j$, we have from
\eqref{key3est},
\begin{equation}\label{finalpositiverec}
P_{f(t_j)}^{bx^\gamma;\text{Ref}\leftarrow:f(t_{j+1})}(\sigma^{(j)}_{l_j}>t_{j+1}-t_j)\le\exp(-\frac{D_j^2}2e^j)
\exp\big(j^{\frac\gamma{1+\gamma}}(\log j)e^{\frac{2bc^{1+\gamma}}{1+\gamma}j}\log 2M\big),
\end{equation}
where, as noted after \eqref{finalsigmaj}, $D_j$ is equal to a postive constant  if $\gamma\ge0$,
and $D_j$ is on the order $j^{\frac\gamma{1+\gamma}}$, if $\gamma\in(-1,0)$.
Thus, \eqref{keyposrec2} follows from \eqref{finalpositiverec} if
$\frac{2bc^{1+\gamma}}{1+\gamma}<1$
\hfill $\square$

\section{Proof of Theorem \ref{liminf}}
As in the proof of Theorem \ref{1}, we can assume that $b$ and $f$ satisfy \eqref{bfassump}.
We will first show that
\begin{equation}\label{lowerbd}
\liminf_{t\to\infty}\frac{X(t)}{f(t)}\le  \rho\ \ \text{a.s.}, \ \text{for any}\
\rho>\big(1-\frac{1+\gamma}{2bc^{1+\gamma}}\big)^\frac1{1+\gamma}.
\end{equation}
The proof of \eqref{lowerbd}
is just a small variant of the proof of recurrence in Theorem \ref{1}; that is, part (i) of Theorem \ref{1}.
As in that proof, let $t_j=e^j$. Recalling the definition of $j_0$ appearing at the very beginning of the proof
of part (i) of Theorem \ref{1}, it follows from \eqref{bfassump} that  $f(t_j)=cj^{\frac1{1+\gamma}}$, for $j\ge j_0$.
In that proof, for $j\ge j_0$, $A_{j+1}$ was defined as the event that the process hits 1 at some time
$t\in[t_j,t_{j+1}]$. For the present proof, we define instead, for each $\rho\in(0,1)$, the event
$A^{(\rho)}_{j+1}$  that the process $X(t)$ satisfies
$X(t)\le \rho f(t_j)$ for some $t\in[t_j,t_{j+1}]$.
We  mimic the proof of Theorem \ref{1}-i up through
\eqref{phiasymp}, using $A^{(\rho)}_{j+1}$ in place of
$A_{j+1}$, replacing the stopping time $T_1$
by the stopping time $T_{\rho f(t_j)}$, and replacing $\phi(1)$ by $\phi(\rho f(t_j))$.
Instead of \eqref{lfto0}, we obtain
\begin{equation}\label{lfto0new}
\begin{aligned}
&1-\big(P_{f(t_j)}^{bx^\gamma;\text{Ref}\leftarrow:f(t_{j+1})}(T_{f(t_{j+1})}<T_{\rho f(t_j)})\big)^{l_j}
\ge\frac12l_j\frac{\phi(f(t_j))-\phi(f(t_{j+1}))}{\phi(\rho f(t_j))-\phi(f(t_{j+1}))}, \\
& \text{for sufficiently large}\ j,\
 \text{if}\
\lim_{j\to\infty}l_j\frac{\phi(f(t_j))}{\phi(\rho f(t_j))}=0.
\end{aligned}
\end{equation}
Instead of \eqref{asympphidiff}, we have
\begin{equation}\label{asympphidiffnew}
\begin{aligned}
&\frac{\phi(f(t_j))-\phi(f(t_{j+1}))}{\phi(\rho f(t_j))-\phi(f(t_{j+1}))}\ge K_1\frac{\phi(f(t_j))}
{\phi(\rho f(t_j))}
\ge K_2\exp\big(-\frac{2bc^{1+\gamma}}{1+\gamma}(1-\rho^{1+\gamma})j\big),\\
& \text{for sufficiently large}\ j,
\end{aligned}
\end{equation}
for constants $K_1,K_2>0$.
From \eqref{lfto0new} and \eqref{asympphidiffnew}, it follows that
\eqref{key1} with $T_1$ replaced by $T_{\rho f(t_j)}$ will hold if we define $l_j\in\mathbb{N}$ by
\begin{equation}\label{ljselectionnew}
l_j=[\frac1{j\log j}
\exp\big(\frac{2bc^{1+\gamma}}{1+\gamma}(1-\rho^{1+\gamma})j\big)],
\end{equation}
since then the general term,
$1-\big(P_{f(t_j)}^{bx^\gamma;\text{Ref}\leftarrow:f(t_{j+1})}(T_{f(t_{j+1})}<T_{\rho f(t_j)})\big)^{l_j}$, will be on the
order at least $\frac 1{j\log j}$.

We now continue to mimic the proof of Theorem \ref{1}-i, starting from the paragraph after
\eqref{ljselection} and up through \eqref{key3est}.
We then insert the present $l_j$ from \eqref{ljselectionnew} in \eqref{key3est} to obtain
\begin{equation}\label{finalsigmajnew}
\begin{aligned}
&P_{f(t_j)}^{bx^\gamma;\text{Ref}\leftarrow:f(t_{j+1})}
(\sigma^{(j)}_{l_j}>t_{j+1}-t_j)\le \exp(-\frac{D_j^2}2e^j)(2M)^{\frac1{j\log j}
\exp\big(\frac{2bc^{1+\gamma}}{1+\gamma}(1-\rho^{1+\gamma})j\big)}=\\
& \exp(-\frac{D_j^2}2e^j)\exp\Big(\frac1{j\log j}e^{\frac{2bc^{1+\gamma}}{1+\gamma}(1-\rho^{1+\gamma})j}\log 2M\Big),\
\text{for sufficiently large}\ j.
\end{aligned}
\end{equation}
Recalling that $D_j$ is equal to a positive constant, if $\gamma\ge0$,  and that $D_j$
is on the order $j^\frac{\gamma}{1+\gamma}$, if $\gamma<0$, it follows
that the right hand side of \eqref{finalsigmajnew} is summable in $j$ if $\frac{2bc^{1+\gamma}}{1+\gamma}(1-\rho^{1+\gamma})<1$, or equivalently,
if $\rho>\big(1-\frac{1+\gamma}{2bc^{1+\gamma}}\big)^\frac1{1+\gamma}$.
Analogous to the proof of Theorem \ref{1}, we conclude then that $P_1(A^{(\rho)}_j\ \text{i.o.})=1$ for $\rho$ as above.
From the definition of $A^{(\rho)}_j$ and the fact that
$f$ is increasing, we conclude that \eqref{lowerbd} holds.

To complete the proof of Theorem \ref{liminf}, we will prove
that
\begin{equation}\label{upperbd}
\liminf_{t\to\infty}\frac{X(t)}{f(t)}\ge \rho\ \ \text{a.s.},
\ \text{for any}\ \rho<\big(1-\frac{1+\gamma}{2bc^{1+\gamma}}\big)^\frac1{1+\gamma}.
\end{equation}
For this direction, we will need some new ingredients.
Recalling  again the definition of $j_0$ appearing at the very beginning of the proof
of part (i) of Theorem \ref{1}, it follows from \eqref{bfassump}
that $f(t)=c(\log t)^{\frac1{1+\gamma}}$ for $t\ge e^{j_0}$.
Let $\tau_1=\inf\{t\ge e^{j_0}: X(t)=f(t)\}$, and for $j\ge2$, let
 $\tau_j=\inf\{t\ge\tau_{j-1}+1:X(t)=f(t)\}$.
 By the remarks in the paragraph preceding Theorem \ref{liminf},
it follows that $\tau_j<\infty$ a.s. $[P_1]$, for all $j$.
By construction, we have
\begin{equation}\label{jbound}
\tau_j>j,\ \text{for all}\ j\ge 1.
\end{equation}

Let $\epsilon\in(0,1)$ and let $\rho\in(0,1)$.
Define $s=s(t)$ by
\begin{equation}\label{s(t)}
\rho(1-2\epsilon)f(s)=\rho(1-\epsilon)f(t),\
\text{for}\ t\ge e^{j_0}.
\end{equation}
Since $f(t)=c(\log t)^{\frac1{1+\gamma}}$, we have
\begin{equation}\label{s(t)again}
s(t)=t^{(\frac{1-\epsilon}{1-2\epsilon})^{1+\gamma}}.
\end{equation}
Of course, $f(s(t))=\frac{1-\epsilon}{1-2\epsilon}f(t)$.
For $j\ge j_0$, define $B_j$ to be the event that the following three inequalities hold:

\noindent i. $X(t)\ge (1-\epsilon)f(\tau_j),\ \tau_j\le t\le \tau_j+1$;

\noindent ii. $X(t)\ge\rho(1-\epsilon)f(\tau_j),\ \tau_j+1\le t\le \tau_{j+1}$;

\noindent iii. $\tau_{j+1}\le s(\tau_j)$.

\noindent (We have suppressed the dependence of $B_j$ on $\epsilon$ and $\rho$.)
\noindent It follows from \eqref{s(t)} that on the event $B_j$ one has
$X(t)\ge (1-2\epsilon)\rho f(t)$, for all $t\in [\tau_j,\tau_{j+1}]$.
Thus, for any $N$, on the event $\cap_{j=N}^\infty B_j$, one has
$\liminf_{t\to\infty}\frac{X(t)}{f(t)}\ge (1-2\epsilon)\rho$.
We will complete the proof of \eqref{upperbd} by showing that
\begin{equation}\label{finaltask}
\lim_{M\to\infty}P_1(\cap_{j=M}^\infty B_j)=1,
\end{equation}
for  all $\rho<\big(1-\frac{1+\gamma}{2bc^{1+\gamma}}\big)^\frac1{1+\gamma}$ and all sufficiently small
$\epsilon$ (depending on $\rho$).

We write
\begin{equation}\label{Bconditioning}
P_1(\cap_{j=M}^NB_j)=\prod_{j=M}^NP_1(B_j|\cap_{i=M}^{j-1}B_i),
\end{equation}
where $\cap_{i=M}^{M-1}B_i$ denotes the entire probability space.
Let
$$
C_j=\{X(t)\ge (1-\epsilon)f(\tau_j),\ \tau_j\le t\le \tau_j+1\}.
$$
 (Note that $C_j$ depends
on the random variable $\tau_j$.)
Let $P_{(1-\epsilon)f(\tau_j)}^{bx^\gamma}$ denote probabilities for the diffusion process corresponding
to $L_{bx^\gamma}$ without reflection, starting from $(1-\epsilon)f(\tau_j)$.
Noting that if $\tau_{j+1}\le s(\tau_j)$, then $X(\tau_{j+1})=f(\tau_{j+1})\le f(s(\tau_j))$, it
follows by the strong Markov property and comparison that
\begin{equation}\label{Bprob}
\begin{aligned}
&P_1(B_j|\cap_{i=M}^{j-1}B_i, \tau_j)\ge\\
& P_{(1-\epsilon)f(\tau_j)}^{bx^\gamma}(T_{\rho(1-\epsilon)f(\tau_j)}>T_{f(s(\tau_j))},\ T_{f(s(\tau_j))}\le s(\tau_j)-\tau_j-1)-
P_1(C_j^c|\tau_j).
\end{aligned}
\end{equation}
Also,
\begin{equation}\label{withoutC}
\begin{aligned}
&P_{(1-\epsilon)f(\tau_j)}^{bx^\gamma}(T_{\rho(1-\epsilon)f(\tau_j)}>T_{f(s(\tau_j))},\ T_{f(s(\tau_j))}\le s(\tau_j)-\tau_j-1)=\\
&P_{(1-\epsilon)f(\tau_j)}^{bx^\gamma}(T_{\rho(1-\epsilon)f(\tau_j)}>T_{f(s(\tau_j))},\ T_{f(s(\tau_j))}\wedge
T_{\rho(1-\epsilon)f(\tau_j)}\le s(\tau_j)-\tau_j-1)\ge\\
&P_{(1-\epsilon)f(\tau_j)}^{bx^\gamma}(T_{\rho(1-\epsilon)f(\tau_j)}>T_{f(s(\tau_j))})-\\
&P_{(1-\epsilon)f(\tau_j)}^{bx^\gamma}(T_{f(s(\tau_j))}\wedge
T_{\rho(1-\epsilon)f(\tau_j)}\ge s(\tau_j)-\tau_j-1).
\end{aligned}
\end{equation}
In order to get a lower bound on $P_1(B_j|\cap_{i=M}^{j-1}B_i, \tau_j)$, we will bound
$P_1(C_j^c|\tau_j)$ and $P_{(1-\epsilon)f(\tau_j)}^{bx^\gamma}(T_{f(s(\tau_j))}\wedge
T_{\rho(1-\epsilon)f(\tau_j)}\ge s(\tau_j)-\tau_j-1)$ from above, and we will calculate the asymptotic behavior of
$P_{(1-\epsilon)f(\tau_j)}^{bx^\gamma}(T_{\rho(1-\epsilon)f(\tau_j)}>T_{f(s(\tau_j))})$.

We start with $P_1(C_j^c|\tau_j)$.
Let $P_0^{\text{BM}}$ denote probabilities for a standard Brownian motion starting from 0,
and let $\hat T_x=\min(T_x, T_{-x})$, for $x>0$.
By the strong Markov property and comparison  we clearly have
\begin{equation}\label{Cc}
P_1(C_j^c|\tau_j)\le P_0^{\text{BM}}(\hat T_{\epsilon f(\tau_j)}\le 1).
\end{equation}
From \cite[Theorem 2.2.2]{P}, we have
$P_0^{\text{BM}}(\hat T_x\le t)\le 2\exp(-\frac{x^2}{2t})$. Thus
from \eqref{Cc} we obtain
\begin{equation}\label{Ccest}
P_1(C_j^c|\tau_j)\le 2\exp\big(-\frac12\epsilon^2 c^2(\log \tau_j)^{\frac2{1+\gamma}}\big).
\end{equation}

We now turn to $P_{(1-\epsilon)f(\tau_j)}^{bx^\gamma}(T_{f(s(\tau_j))}\wedge
T_{\rho(1-\epsilon)f(\tau_j)}\ge s(\tau_j)-\tau_j-1)$.
Denote by $Y(t)$ the diffusion corresponding to the operator
$\frac12\frac{d^2}{dx^2}+bx^\gamma\frac d{dx}$ and the measure $P_{(1-\epsilon)f(\tau_j)}^{bx^\gamma}$, and denote by
$\mathcal{W}(t)$ standard Brownian motion
  on $\big(\rho(1-\epsilon)f(\tau_j),\infty)$
 with reflection at the  endpoint.
Denote probabilities for this Brownian motion starting from $x$ by $P_x^{0;\text{Ref}\rightarrow:\rho(1-\epsilon)f(\tau_j)}$.
(The superscript 0 signifies 0 drift.) Since the drift of the $Y$ diffusion is positive,
we can couple $Y(t)$ and $\mathcal{W}(t)$ so that $\mathcal{W}(t)\le Y(t)$, for all $t\in[0,T_{\rho(1-\epsilon)f(\tau_j)}\wedge
T_{f(s(\tau_j))}]$, where $T_{\rho(1-\epsilon)f(\tau_j)}$
and $T_{f(s(\tau_j))}$ refer to the hitting times for the $Y$ process.
(Note that we have been using the generic $T_a$ for the hitting time of $a$ for any process, the process
in question being inferred from the probability measure which appears with it.)
Thus, for any $t>0$,
\begin{equation}\label{hittingcomp}
P_{(1-\epsilon)f(\tau_j)}^{bx^\gamma}(T_{f(s(\tau_j))}\wedge
T_{\rho(1-\epsilon)f(\tau_j)}\ge t)\le
P_{(1-\epsilon)f(\tau_j)}^{0;\text{Ref}\rightarrow:\rho(1-\epsilon)f(\tau_j)}(T_{f(s(\tau_j))}\ge t).
\end{equation}
For ease of notation, in the analysis below, we let
$L_1=\rho(1-\epsilon)f(\tau_j)$, $L_2=(1-\epsilon)f(\tau_j)$ and $L_3=f(s(\tau_j))$.
Let $P_x^{\text{BM}}$ denote probabilities for a standard Brownian motion starting from $x$.
By the isotropy of Brownian motion and the fact that a reflected Brownian motion can be realized as the absolute
value of a Brownian motion, we have
\begin{equation}\label{isotr}
P_{(1-\epsilon)f(\tau_j)}^{0;\text{Ref}\rightarrow:\rho(1-\epsilon)f(\tau_j)}(T_{f(s(\tau_j))}\ge t)=
P^{\text{BM}}_{L_2-L_1}(T_{L_3-L_1}\wedge T_{-(L_3-L_1)}\ge t).
\end{equation}
Using Brownian scaling for the first inequality and symmetry for the second one,
we have
\begin{equation}\label{scaling}
\begin{aligned}
&P^{\text{BM}}_{L_2-L_1}(T_{L_3-L_1}\wedge T_{-(L_3-L_1)}\ge t)=P^{\text{BM}}_{\frac{L_2-L_1}{L_3-L_1}}(T_1\ge \frac t{(L_3-L_1)^2})\le\\
&P^{\text{BM}}_0(T_1\ge \frac t{(L_3-L_1)^2}).
\end{aligned}
\end{equation}
As is well-known, there exist $\kappa,\lambda>0$ such that
$P^{\text{BM}}_0(T_1\ge t)\le \kappa e^{-\lambda t}$, for all $t\ge0$.
Thus, from \eqref{hittingcomp}-\eqref{scaling}, choosing $t=s(\tau_j)-\tau_j-1$,  we conclude that
\begin{equation}\label{finalfromBM}
P_{(1-\epsilon)f(\tau_j)}^{bx^\gamma}(T_{f(s(\tau_j))}\wedge
T_{\rho(1-\epsilon)f(\tau_j)}\ge s(\tau_j)-\tau_j-1)\le
\kappa \exp\big(-\frac {\lambda (s(\tau_j)-\tau_j-1)}{(L_3-L_1)^2}\big).
\end{equation}

We now calculate the asymptotic behavior
of  $P_{(1-\epsilon)f(\tau_j)}^{bx^\gamma}(T_{\rho(1-\epsilon)f(\tau_j)}>T_{f(s(\tau_j))})$
via that of  $P_{(1-\epsilon)f(\tau_j)}^{bx^\gamma}(T_{\rho(1-\epsilon)f(\tau_j)}<T_{f(s(\tau_j))})$.
Similar to \eqref{pottheory}, we have
\begin{equation}\label{phithm4}
P_{(1-\epsilon)f(\tau_j)}^{bx^\gamma}(T_{\rho(1-\epsilon)f(\tau_j)}<T_{f(s(\tau_j))})=
\frac{\phi(f(s(\tau_j)))-\phi((1-\epsilon)f(\tau_j))}{\phi(f(s(\tau_j)))-\phi(\rho(1-\epsilon)f(\tau_j))}.
\end{equation}
In light of  \eqref{s(t)again} and
\eqref{phiasymp}, it follows from \eqref{phithm4} that
\begin{equation}\label{asympsmall}
P_{(1-\epsilon)f(\tau_j)}^{bx^\gamma}(T_{\rho(1-\epsilon)f(\tau_j)}<T_{f(s(\tau_j))})\sim
\frac{\phi((1-\epsilon)f(\tau_j))}{\phi(\rho(1-\epsilon)f(\tau_j))},\ \text{as}\ \tau_j\to\infty.
\end{equation}
From \eqref{phiasymp} and the fact that $f(t)=c(\log t)^\frac1{1+\gamma}$, we have
\begin{equation}\label{intermediate}
\frac{\phi((1-\epsilon)f(\tau_j))}{\phi(\rho(1-\epsilon)f(\tau_j))}
\sim \rho^\gamma\tau_j^{-\frac{2bc^{1+\gamma}(1-\epsilon)^{1+\gamma}}{1+\gamma}(1-\rho^{1+\gamma})},
\ \text{as}\ \tau_j\to\infty.
\end{equation}
Thus, from \eqref{intermediate} and \eqref{asympsmall},
\begin{equation}\label{asympsmallrate}
P_{(1-\epsilon)f(\tau_j)}^{bx^\gamma}(T_{\rho(1-\epsilon)f(\tau_j)}<T_{f(s(\tau_j))})\sim
\rho^\gamma\tau_j^{-\frac{2bc^{1+\gamma}(1-\epsilon)^{1+\gamma}}{1+\gamma}(1-\rho^{1+\gamma})},
\ \text{as}\ \tau_j\to\infty.
\end{equation}

From \eqref{Bprob}, \eqref{withoutC}, \eqref{Ccest}, \eqref{finalfromBM}, \eqref{asympsmallrate} and \eqref{jbound}
we have
\begin{equation}\label{finalest}
\begin{aligned}
&P_1(B_j|\cap_{i=M}^{j-1}B_i, \tau_j)\ge1-2\rho^\gamma\tau_j^{-\frac{2bc^{1+\gamma}(1-\epsilon)^{1+\gamma}}{1+\gamma}
(1-\rho^{1+\gamma})}-\\
&2\exp\big(-\frac12\epsilon^2 c^2(\log \tau_j)^{\frac2{1+\gamma}}\big)-
\kappa \exp\big(-\frac {\lambda (s(\tau_j)-\tau_j-1)}{(L_3-L_1)^2}\big),\\
&\text{for sufficiently large}\ j.
\end{aligned}
\end{equation}
Since
$$
\frac{s(\tau_j)-\tau_j-1}{(L_3-L_1)^2}=\frac{\tau_j^{(\frac{1-\epsilon}{1-2\epsilon})^{1+\gamma}}-\tau_j-1}
{(\frac1{1-2\epsilon}-\rho)^2(1-\epsilon)^2c^2(\log \tau_j)^{\frac2{1+\gamma}}},
$$
and since $\tau_j>j$ by \eqref{jbound}, it follows that  \eqref{finalest} holds with
$\tau_j$ replaced by $j$ on the right hand side of the inequality. Then taking the conditional expectation
with respect to $\cap_{i=M}^{j-1}B_i$  on the left hand side, to remove   $\tau_j$ from the conditioning there,
we conclude that
\begin{equation}\label{finalestagain}
\begin{aligned}
&P_1(B_j|\cap_{i=M}^{j-1}B_i)\ge
1-2\rho^\gamma j^{-\frac{2bc^{1+\gamma}(1-\epsilon)^{1+\gamma}}{1+\gamma}
(1-\rho^{1+\gamma})}-\\
&2\exp\big(-\frac12\epsilon^2 c^2(\log j)^{\frac2{1+\gamma}}\big)-
\kappa \exp\big(-\frac{j^{(\frac{1-\epsilon}{1-2\epsilon})^{1+\gamma}}-j-1}
{(\frac1{1-2\epsilon}-\rho)^2(1-\epsilon)^2c^2(\log j)^{\frac2{1+\gamma}}}\big),\\
&\text{for sufficiently large}\ j.
\end{aligned}
\end{equation}

Clearly, $\sum_{j=j_0}^\infty\exp\big(-\frac{j^{(\frac{1-\epsilon}{1-2\epsilon})^{1+\gamma}}-j-1}
{(\frac1{1-2\epsilon}-\rho)^2(1-\epsilon)^2c^2(\log j)^{\frac2{1+\gamma}}}\big)<\infty$.
Since by assumption  $\gamma<1$, it follows that
$\sum_{j=j_0}^\infty \exp\big(-\frac12\epsilon^2 c^2(\log j)^{\frac2{1+\gamma}}\big)<\infty$.
 Since $\frac{2bc^{1+\gamma}}{1+\gamma}(1-\rho^{1+\gamma})>1$ is equivalent
to $\rho<\big(1-\frac{1+\gamma}{2bc^{1+\gamma}}\big)^\frac1{1+\gamma}$,
 we also have
$\sum_{j=j_0}^\infty j^{-\frac{2bc^{1+\gamma}(1-\epsilon)^{1+\gamma}}{1+\gamma}(1-\rho^{1+\gamma})}<\infty$,
for all $\rho<\big(1-\frac{1+\gamma}{2bc^{1+\gamma}}\big)^\frac1{1+\gamma}$ and sufficiently small
$\epsilon$ (depending on $\rho$).
Using these facts with \eqref{finalestagain}
and \eqref{Bconditioning},
we conclude that
$$
\lim_{M\to\infty}\lim_{N\to\infty}P_1(\cap_{j=M}^NB_j)=1,
$$
which gives \eqref{finaltask}
for  all $\rho<\big(1-\frac{1+\gamma}{2bc^{1+\gamma}}\big)^\frac1{1+\gamma}$ and all sufficiently small
$\epsilon$ (depending on $\rho$).
\hfill $\square$
\section{Proof of Theorem \ref{liminfagain}}

\noindent \it Proof of  (i)\rm. The proof is almost exactly the same as the proof of Theorem \ref{liminf}
starting from \eqref{upperbd}, using
$(\log t)^l$ instead of $(\log t)^\frac1{1+\gamma}$ (and with $b=c=1$). The one place in the proof where this results
in a meaningful difference is in the estimate on
$P_{(1-\epsilon)f(\tau_j)}^{x^\gamma}(T_{\rho(1-\epsilon)f(\tau_j)}<T_{f(s(\tau_j))})$.
We have \eqref{asympsmall} as in the proof of Theorem \ref{liminf}.
From \eqref{phiasymp} and the fact that $f(t)=(\log t)^l$, we have, instead of \eqref{intermediate},
$$
\frac{\phi((1-\epsilon)f(\tau_j))}{\phi(\rho(1-\epsilon)f(\tau_j))}
\sim \rho^\gamma\exp\Big(-\frac2{1+\gamma}(1-\epsilon)^{1+\gamma}(1-\rho^{1+\gamma})(\log \tau_j)^{l(1+\gamma)}\Big).
$$
Using this with \eqref{asympsmall} gives, instead of \eqref{asympsmallrate},
$$
\begin{aligned}
&P_{(1-\epsilon)f(\tau_j)}^{bx^\gamma}(T_{\rho(1-\epsilon)f(\tau_j)}<T_{f(s(\tau_j))})\sim\\
&\rho^\gamma\exp\Big(-\frac2{1+\gamma}(1-\epsilon)^{1+\gamma}(1-\rho^{1+\gamma})(\log \tau_j)^{l(1+\gamma)}\Big),
\ \text{as}\ \tau_j\to\infty.
\end{aligned}
$$
By \eqref{jbound},  $\tau_j>j$. Since $l(1+\gamma)>1$, the right hand side above  is summable for all $\rho\in(0,1)$.
The proof of Theorem \ref{liminf} then gives $\liminf_{t\to\infty}\frac{X(t)}{f(t)}\ge\rho$ a.s.,
for all $\rho\in(0,1)$; thus, $\liminf_{t\to\infty}\frac{X(t)}{f(t)}\ge1$ a.s. Since $X(t)\le f(t)$, we conclude that
$\lim_{t\to\infty}\frac{X(t)}{f(t)}=1$ a.s.
\medskip

\noindent \it Proof of (ii)\rm. We first prove \eqref{q0role}. We follow the same kind of strategy used to prove \eqref{upperbd} in the proof of Theorem \ref{liminf}.
Let $\{\tau_j\}_{j=1}^\infty$ be defined as it is following \eqref{upperbd}.
Let $\epsilon\in(0,1)$ and  $q\in(0,l)$. If we were to define $s(t)$, for $t\ge1$,  by
$$
f(s)-(1+\epsilon)s^q=f(t)-t^q \ (\text{equivalently}, s^l-(1+\epsilon)s^q=t^l-t^q),
$$
 we would have
$s^l=t^l+(1+\epsilon)s^q-t^q\ge t^l+\epsilon t^q$. In light of this, we define for simplicity $s=s(t)$ by
$$
s^l=t^l+\epsilon t^q.
$$
 Then
\begin{equation}\label{f(s(t))}
\begin{aligned}
&f(s(t))=(s(t))^l=t^l+\epsilon t^q;\\
&s(t)=t(1+\epsilon t^{q-l})^\frac1l= t+\epsilon t^{1+q-l}+\text{lower order terms}  ,\ \text{as}\ t\to\infty.
\end{aligned}
\end{equation}
and
\begin{equation}\label{f(s(t))-}
\begin{aligned}
&f(s(t))-(1+\epsilon)(s(t))^q=(s(t))^l-(1+\epsilon)(s(t))^q
=\\
&t^l+\epsilon t^q-(1+\epsilon)(s(t))^q\le t^l-t^q=
f(t)-t^q.
\end{aligned}
\end{equation}
Let $B_j$ be the event that

\noindent i. $X(t)\ge f(\tau_j)-\epsilon\tau_j^q,\ \tau_j\le t\le \tau_j+1$;

\noindent ii. $X(t)\ge f(\tau_j)-\tau_j^q,\ \tau_j+1\le t\le \tau_{j+1}$;

\noindent iii. $\tau_{j+1}\le s(\tau_j)$.

\noindent (We have suppressed the dependence of $B_j$ on $\epsilon$ and $q$.)
\noindent It follows from \eqref{f(s(t))-} that on the event $B_j$ one has
$X(t)\ge f(t)-(1+\epsilon)t^q$, for all $t\in[\tau_j,\tau_{j+1}]$.
Thus, for any $N$, on the event $\cap_{j=N}^\infty B_j$, one has
$$
\limsup_{t\to\infty}\big(f(t)-X(t)-(1+\epsilon)t^q\big)\le0.
$$
Therefore, the proof of \eqref{q0role}   will be completed when we show that
\begin{equation}\label{finaltaskagain}
\lim_{M\to\infty}P_1(\cap_{j=M}^\infty B_j)=1,
\end{equation}
for some $\epsilon\in(0,1)$ and all $q>q_0$.

We write
\begin{equation}\label{Bconditioningagain}
P_1(\cap_{j=M}^NB_j)=\prod_{j=M}^NP_1(B_j|\cap_{i=M}^{j-1}B_i),
\end{equation}
where $\cap_{i=M}^{M-1}B_i$ denotes the entire probability space.
Let
$$
C_j=\{X(t)\ge f(\tau_j)-\epsilon\tau_j^q,\ \tau_j\le t\le \tau_j+1\}.
$$
 (Note that $C_j$ depends
on the random variable $\tau_j$.)
Let $P_{f(\tau_j)-\epsilon \tau_j^q}^{x^\gamma}$ denote probabilities for the diffusion process corresponding
to $L_{x^\gamma}$ without reflection, starting from $f(\tau_j)-\epsilon \tau_j^q$.
Noting that if $\tau_{j+1}\le s(\tau_j)$, then $X(\tau_{j+1})=f(\tau_{j+1})\le f(s(\tau_j))$, it
follows by the strong Markov property and comparison that
\begin{equation}\label{Bprobagain}
\begin{aligned}
&P_1(B_j|\cap_{i=M}^{j-1}B_i, \tau_j)\ge \\
&P_{f(\tau_j)-\epsilon \tau_j^q}^{x^\gamma}(T_{f(\tau_j)-\tau_j^q}>T_{f(s(\tau_j))},\
 T_{f(s(\tau_j))}\le s(\tau_j)-\tau_j-1)-P_1(C_j^c|\tau_j).
\end{aligned}
\end{equation}
Also,
\begin{equation}\label{withoutCagain}
\begin{aligned}
&P_{f(\tau_j)-\epsilon \tau_j^q}^{x^\gamma}(T_{f(\tau_j)-\tau_j^q}>T_{f(s(\tau_j))},\
 T_{f(s(\tau_j))}\le s(\tau_j)-\tau_j-1)=\\
&P_{f(\tau_j)-\epsilon \tau_j^q}^{x^\gamma}(T_{f(\tau_j)-\tau_j^q}>T_{f(s(\tau_j))},\
 T_{f(s(\tau_j))}\wedge T_{f(\tau_j)-\tau_j^q}\le s(\tau_j)-\tau_j-1)
\ge\\
&P_{f(\tau_j)-\epsilon \tau_j^q}^{x^\gamma}(T_{f(\tau_j)-\tau_j^q}>T_{f(s(\tau_j))})-\\
&P_{f(\tau_j)-\epsilon \tau_j^q}^{x^\gamma}(T_{f(s(\tau_j))}\wedge
T_{f(\tau_j)-\tau_j^q}\ge s(\tau_j)-\tau_j-1).
\end{aligned}
\end{equation}
In order to get a lower bound on $P_1(B_j|\cap_{i=M}^{j-1}B_i, \tau_j)$, we will bound
$P_1(C_j^c|\tau_j)$ and $P_{f(\tau_j)-\epsilon \tau_j^q}^{x^\gamma}(T_{f(s(\tau_j))}\wedge
T_{f(\tau_j)-\tau_j^q}\ge s(\tau_j)-\tau_j-1)$ from above, and we will calculate the asymptotic behavior of
$P_{f(\tau_j)-\epsilon \tau_j^q}^{x^\gamma}(T_{f(\tau_j)-\tau_j^q}>T_{f(s(\tau_j))})$.

We start with $P_1(C_j^c|\tau_j)$. We mimic the paragraph
containing \eqref{Cc}, the only change being that $\epsilon f(\tau_j)$ is replaced by $\epsilon \tau_j^q$.
Thus, similar to \eqref{Ccest}, we obtain
\begin{equation}\label{Ccestagain}
P_1(C_j^c|\tau_j)\le 2\exp\big(-\frac12\epsilon^2 \tau_j^{2q}\big).
\end{equation}

We now turn to $P_{f(\tau_j)-\epsilon \tau_j^q}^{x^\gamma}(T_{f(s(\tau_j))}\wedge
T_{f(\tau_j)-\tau_j^q}\ge s(\tau_j)-\tau_j-1)$. We mimic the paragraph following \eqref{Ccest},
the only changes being that $(1-\epsilon)f(\tau_j)$ is replaced by $f(\tau_j)-\epsilon \tau_j^q$,
$\rho(1-\epsilon)f(\tau_j)$ is replaced by $f(\tau_j)-\tau_j^q$ and $b$ is set to 1.
Similar to \eqref{finalfromBM}, we obtain,
\begin{equation}\label{finalfromBMagain}
\begin{aligned}
&P_{f(\tau_j)-\epsilon \tau_j^q}^{x^\gamma}(T_{f(s(\tau_j))}\wedge
T_{f(\tau_j)-\tau_j^q}\ge s(\tau_j)-\tau_j-1)\le
\kappa \exp\big(-\frac {\lambda (s(\tau_j)-\tau_j-1)}{(L_3-L_1)^2}\big),\\
&\text{where}\ L_3=f(s(\tau_j))\ \text{and}\ L_1=f(\tau_j)-\tau_j^q.
\end{aligned}
\end{equation}

We now calculate the asymptotic behavior
of  $P_{f(\tau_j)-\epsilon \tau_j^q}^{x^\gamma}(T_{f(\tau_j)-\tau_j^q}>T_{f(s(\tau_j))})$.
via that of  $P_{f(\tau_j)-\epsilon \tau_j^q}^{x^\gamma}(T_{f(\tau_j)-\tau_j^q}<T_{f(s(\tau_j))})$.
Similar to \eqref{pottheory}, we have
\begin{equation}\label{phithm5}
P_{f(\tau_j)-\epsilon \tau_j^q}^{x^\gamma}(T_{f(\tau_j)-\tau_j^q}<T_{f(s(\tau_j))})=
\frac{\phi(f(s(\tau_j)))-\phi(f(\tau_j)-\epsilon\tau_j^q)}{\phi(f(s(\tau_j)))-\phi(f(\tau_j)-\tau_j^q)}.
\end{equation}
For any $\lambda\in\mathbb{R}$, we have
\begin{equation}\label{qq0}
(t^l+\lambda t^q)^{1+\gamma}=t^{l(1+\gamma)}+\lambda(1+\gamma)t^{l\gamma+q}+\ \text{lower order terms
as}\ t\to\infty.
\end{equation}
We now make the assumption, as in the statement of the theorem, that $q>q_0$. Thus, $l\gamma+q>0$.
Using this  in \eqref{qq0}, along with \eqref{phiasymp} (with $b=1$) and \eqref{phithm5},
and recalling that $f(\tau_j)=\tau_j^l$ and that, from \eqref{f(s(t))}, $f(s(\tau_j))=\tau_j^l+\epsilon \tau_j^q$,
 we conclude that
\begin{equation}\label{asympsmallrateagain}
P_{f(\tau_j)-\epsilon \tau_j^q}^{x^\gamma}(T_{f(\tau_j)-\tau_j^q}<T_{f(s(\tau_j))})\sim
\exp(-2(1-\epsilon)\tau_j^{l\gamma+q}), \ \text{as}\ \tau_j\to\infty.
\end{equation}

From \eqref{Bprobagain}-\eqref{finalfromBMagain} and \eqref{asympsmallrateagain}, we have
\begin{equation*}
\begin{aligned}
&P_1(B_j|\cap_{i=M}^{j-1}B_i, \tau_j)\ge1-2\exp(-2(1-\epsilon)\tau_j^{l\gamma+q})-\\
&2\exp\big(-\frac12\epsilon^2 \tau_j^{2q}\big)-
\kappa \exp\big(-\frac {\lambda (s(\tau_j)-\tau_j-1)}{(L_3-L_1)^2}\big),\ \text{for sufficiently large}\ \tau_j.
\end{aligned}
\end{equation*}
From \eqref{f(s(t))}, we have
$s(\tau_j)-\tau_j-1\ge\frac\epsilon2\tau_j^{1+q-l}$ for large $\tau_j$.
From \eqref{finalfromBMagain} and \eqref{f(s(t))}, we have
$L_3-L_1=f(s(\tau_j))-f(\tau_j)+\tau_j^q=(1+\epsilon)\tau_j^q$.
Thus, for large $\tau_j$,
$$
\frac {s(\tau_j)-\tau_j-1}{(L_3-L_1)^2}\ge\frac{\epsilon}{2(1+\epsilon)^2}\tau_j^{1-q-l}.
$$
If $1-q-l>0$, then we can complete the proof just like we completed the proof of
Theorem \ref{liminf} and conclude that \eqref{finaltaskagain} holds, and thus that \eqref{q0role} holds.
Note that in order to come to this conclusion, we have needed to assume
that $q>q_0=\max(0,-l\gamma)$ and that $1-q-l>0$; that is,
we need  $\max(0,-l\gamma)<1-l$ and $q\in(\max(0,-l\gamma),1-l)$.
A fundamental assumption in the theorem is that $l\in(0,\frac1{1-\gamma})$. For these values
of $l$, the above inequality always holds. Thus,
\eqref{finaltaskagain} holds for those  $q$ which are larger than $q_0$ and sufficiently close to $q_0$.
Consequently, \eqref{q0role} holds for all $q$ which are larger than $q_0$ and sufficiently close to $q_0$.
However, if  \eqref{q0role} holds for some $q$, then clearly it also holds for all larger $q$.
Thus, \eqref{q0role} holds for all $q>q_0$.

We now turn to the proof of \eqref{q0roleagain}.
We have
 $\gamma\in(-1,0]$ and  $q_0=-\gamma l\in[0,l)$.
Let $t_j=j^k$, for $j\ge1$ and some $k>1$ to be fixed later. For $M>0$, let
 $A^M_{j+1}$ be the event that $X(t)\le f(t_j)-Mt_j^{q_0}$ for some $t\in[t_j,t_{j+1}]$.
Clearly, $\limsup_{t\to\infty}  \frac{f(t)-X(t)}{t^{q_0}}\ge M$ on the event
$\{A^M_j\ \text{i.o.}\}$.
The conditional version of the Borel-Cantelli lemma \cite{D} shows that if
\begin{equation}\label{BC-cond-inftyagain}
\sum_{j=1}^\infty P_1(A^M_{j+1}|\mathcal{F}_{t_j})=\infty,
\end{equation}
then  $P_1(A^M_j\ \text{i.o.})=1$
Thus, to prove \eqref{q0roleagain}, it suffices to show that \eqref{BC-cond-inftyagain} holds for all $M>0$.

Since up to time $t_j$, the largest the process can be is $f(t_j)$, and
since
up to time $t_{j+1}$ the time-dependent domain is contained in $[1,f(t_{j+1})]$, it follows
by comparison that
\begin{equation}\label{Ajcondagain}
P_1(A^M_{j+1}|\mathcal{F}_{t_j})\ge P_{f(t_j)}^{x^\gamma;\text{Ref}\leftarrow:f(t_{j+1})}(T_{f(t_j)-Mt_j^{q_0}}\le t_{j+1}-t_j).
\end{equation}
Clearly,
\begin{equation}\label{doublereflect}
\begin{aligned}
&P_{f(t_j)}^{x^\gamma;\text{Ref}\leftarrow:f(t_{j+1})}(T_{f(t_j)-Mt_j^{q_0}}\le t_{j+1}-t_j)=\\
&P_{f(t_j)}^{x^\gamma;\text{Ref}\leftrightarrow:f(t_j)-Mt_j^{q_0},f(t_{j+1})}(T_{f(t_j)-Mt_j^{q_0}}\le t_{j+1}-t_j),
\end{aligned}
\end{equation}
where $P_{f(t_j)}^{x^\gamma;\text{Ref}\leftrightarrow:f(t_j)-Mt_j^{q_0},f(t_{j+1})}$
corresponds to the  $L_{x^\gamma}$ diffusion with reflection at both
$f(t_j)-Mt_j^{q_0}$ and $f(t_{j+1})$.

We  estimate the right hand side of \eqref{doublereflect}.
We have
$$
\{T_{f(t_j)-Mt_j^{q_0}}<T_{f(t_{j+1})}\}-\{T_{f(t_{j+1})}>t_{j+1}-t_j\}\subset
\{T_{f(t_j)-Mt_j^{q_0}}\le t_{j+1}-t_j\}.
$$
Thus,
\begin{equation}\label{keyidea}
\begin{aligned}
&P_{f(t_j)}^{x^\gamma;\text{Ref}\leftrightarrow:f(t_j)-Mt_j^{q_0},f(t_{j+1})}(T_{f(t_j)-Mt_j^{q_0}}\le t_{j+1}-t_j)\ge\\
&P_{f(t_j)}^{x^\gamma;\text{Ref}\leftrightarrow:f(t_j)-Mt_j^{q_0},f(t_{j+1})}
(T_{f(t_j)-Mt_j^{q_0}}<T_{f(t_{j+1})})-\\
&P_{f(t_j)}^{x^\gamma;\text{Ref}\leftrightarrow:f(t_j)-Mt_j^{q_0},f(t_{j+1})}(  T_{f(t_{j+1})}>t_{j+1}-t_j)=\\
&P_{f(t_j)}^{x^\gamma}(T_{f(t_j)-Mt_j^{q_0}}<T_{f(t_{j+1})})-
P_{f(t_j)}^{x^\gamma;\text{Ref}\rightarrow:f(t_j)-Mt_j^{q_0}}(  T_{f(t_{j+1})}>t_{j+1}-t_j),
\end{aligned}
\end{equation}
where $P_{f(t_j)}^{x^\gamma}$
corresponds to the $L_{x^\gamma}$ diffusion without reflection.

Similar to \eqref{pottheory}, we have
\begin{equation}\label{twosided}
P_{f(t_j)}^{x^\gamma}(T_{f(t_j)-Mt_j^{q_0}}<T_{f(t_{j+1})})=
\frac{\phi(f(t_j))-\phi(f(t_{j+1}))}{\phi(f(t_j)-Mt_j^{q_0})-\phi(f(t_{j+1}))},
\end{equation}
where $\phi$ is as in \eqref{phi} with $b=1$.
We now choose $k$ so that $kl(1+\gamma)>1$. Then
since
$$
(f(t_{j+1}))^{1+\gamma}-(f(t_j))^{1+\gamma}=(j+1)^{kl(1+\gamma)}-j^{kl(1+\gamma)}
$$
 is on the order
$j^{kl(1+\gamma)-1}$,
 it follows from
\eqref{twosided} and \eqref{phiasymp} that
\begin{equation}\label{twosidedsimpler}
P_{f(t_j)}^{x^\gamma}(T_{f(t_j)-Mt_j^{q_0}}<T_{f(t_{j+1})})\sim
\frac{\phi(f(t_j))}{\phi(f(t_j)-Mt_j^{q_0})}.
\end{equation}
Now
\begin{equation}\label{anotherasymp}
\begin{aligned}
&(f(t_j))^{1+\gamma}=j^{kl(1+\gamma)};\\
&(f(t_j)-Mt_j^{q_0})^{1+\gamma}=(j^{kl}-Mj^{-kl\gamma})^{1+\gamma}=j^{kl(1+\gamma)}-M(1+\gamma)+o(1),\ \text{as}\ j\to\infty.
\end{aligned}
\end{equation}
Using \eqref{anotherasymp} and \eqref{phiasymp}, we conclude from \eqref{twosidedsimpler} that
\begin{equation}\label{liminftwosided}
\liminf_{j\to\infty}P_{f(t_j)}^{x^\gamma}(T_{f(t_j)-Mt_j^{q_0}}<T_{f(t_{j+1})})>0.
\end{equation}

We now consider $P_{f(t_j)}^{x^\gamma;\text{Ref}\rightarrow:f(t_j)-Mt_j^{q_0}}(  T_{f(t_{j+1})}>t_{j+1}-t_j)$.
By Markov's inequality, we have for $\lambda>0$,
\begin{equation}\label{Markovagain}
\begin{aligned}
&P_{f(t_j)}^{x^\gamma;\text{Ref}\rightarrow:f(t_j)-Mt_j^{q_0}}(  T_{f(t_{j+1})}>t_{j+1}-t_j)\le\\
&\exp\big(-\lambda (t_{j+1}-t_j)\big)E_{f(t_j)}^{x^\gamma;\text{Ref}\rightarrow:f(t_j)-Mt_j^{q_0}}\exp(\lambda T_{f(t_{j+1})}).
\end{aligned}
\end{equation}
We apply Proposition \ref{4} to the expectation on the right hand side of \eqref{Markovagain}, with
$\alpha=f(t_j)-Mt_j^{q_0}=j^{kl}-Mj^{-kl\gamma}$, $\beta=f(t_{j+1})=(j+1)^{kl}$,  $b=1$ and $\lambda=\bar\lambda$, where
$\bar\lambda$ is as in the statement of the proposition.
Then
\begin{equation}\label{exp2}
E_{f(t_j)}^{x^\gamma;\text{Ref}\rightarrow:f(t_j)-Mt_j^{q_0}}\exp(\bar\lambda T_{f(t_{j+1})})\le2.
\end{equation}
Since $\gamma\le0$, we have $\min(\alpha^\gamma,\beta^\gamma)=\beta^\gamma$; thus,
$$
\bar\lambda=\frac{(j+1)^{kl\gamma}}{(2e-1)\big((j+1)^{kl}-j^{kl}+Mj^{-kl\gamma}\big)}.
$$
Now $(j+1)^{kl}-j^{kl}$ is on the order $j^{kl-1}$. Recall from the previous paragraph that we have chosen $k$ so that
$kl(1+\gamma)>1$. Thus,
  the denominator on the right hand side above
 is on the order $j^{kl-1}$, and thus $\bar\lambda$ is on the order   $j^{kl\gamma-kl+1}$.
 Since $t_{j+1}-t_j=(j+1)^k-j^k$ is on the order $j^{k-1}$, the expression
 $(t_{j+1}-t_j)\bar\lambda$ is on the order $j^{k\big(1-l(1-\gamma)\big)}$.
 By assumption, $l(1-\gamma)<1$; thus $\lim_{j\to\infty}(t_{j+1}-t_j)\bar\lambda=\infty$.
 Using this with \eqref{exp2}, and substituting $\lambda=\bar\lambda$ in \eqref{Markovagain},
 we conclude that
 \begin{equation}\label{convto0}
 \lim_{j\to\infty}P_{f(t_j)}^{x^\gamma;\text{Ref}\rightarrow:f(t_j)-Mt_j^{q_0}}(  T_{f(t_{j+1})}>t_{j+1}-t_j)=0.
 \end{equation}
From \eqref{Ajcondagain}-\eqref{keyidea}, \eqref{liminftwosided} and  \eqref{convto0}
we conclude that \eqref{BC-cond-inftyagain} holds for any $M>0$.
\hfill $\square$

\end{document}